\documentclass[a4paper,12pt]{amsart}

\usepackage{a4wide}
\usepackage{pgf,tikz}
\usepackage{amssymb}
\usepackage{enumerate}

\usepackage{amsmath}

\newif\ifdetails
\detailstrue
\newcommand{\DETAIL}[1]%
{\ifdetails\par\fbox{\begin{minipage}{0.9\linewidth}\textit{Detail:}
      #1\end{minipage}}\par\fi}
\newcommand{\TODO}[1]%
{\ifdetails\par\fbox{\begin{minipage}{0.9\linewidth}\textbf{TODO:}
      #1\end{minipage}}\par\fi}

\usepackage{makecell}
\usepackage{relsize}

\newtheorem{lemma}{Lemma}
\newtheorem{proposition}[lemma]{Proposition}
\newtheorem{theorem}[lemma]{Theorem}

\theoremstyle{remark}

\newtheorem{problem}{Problem}
\DeclareMathOperator{\N}{N}
\DeclareMathOperator{\V}{V}

\usepackage{caption}
\usepackage{subcaption}

\usepackage{float} 
\usepackage[pdftex,a4paper,
citecolor = blue, colorlinks=true,urlcolor=blue]{hyperref}
\usepackage{mathrsfs}
\usetikzlibrary{arrows}

\newcommand{\old}[1]{{}}

\title[Cut and pendant verteices vs. number of connected induced subgraphs]
{Cut and pendant vertices and the number of connected induced subgraphs of a graph}

\author{Audace A. V. Dossou-Olory}

\thanks{The author was partially supported by the National Research Foundation of South Africa, grant 118521.}

\address{Audace A. V. Dossou-Olory \\ Department of Mathematics and Applied Mathematics \\ University of Johannesburg \\ P.O. Box 524, Auckland Park, Johannesburg 2006, South Africa}

\email{audace@aims.ac.za}

\subjclass[2010]{Primary 05C30; secondary 05C35, 05C05}
\keywords{cut vertex, pendant vertex, induced subgraph, connected graph, extremal graph structure, tree}

\begin{document}

\begin{abstract}
A vertex whose removal in a graph $G$ increases the number of components of $G$ is called a cut vertex. For all $n,c$, we determine the maximum number of connected induced subgraphs in a connected graph with order $n$ and $c$ cut vertices, and also characterise those graphs attaining the bound. Moreover, we show that the cycle has the smallest number of connected induced subgraphs among all cut vertex-free connected graphs. The general case $c>0$ remains an open task. We also characterise the extremal graph structures given both order and number of pendant vertices, and establish the corresponding formulas for the number of connected induced subgraphs. The `minimal' graph in this case is a tree, thus coincides with the structure that was given by Li and Wang~[Further analysis on the total number of subtrees of trees. \emph{Electron. J. Comb.} 19(4), \#P48, 2012].
\end{abstract}

\maketitle

\section{Introduction and Preliminaries}
Let $G$ be a simple graph with vertex set $V(G)$ and edge set $E(G)$. The graph $G$ is said to be connected if for all $u,v \in V(G)$, there is a $u-v$ path in $G$. An induced subgraph $H$ of $G$ is a graph such that $\emptyset \neq V(H) \subseteq V(G)$ and $E(H)$ consists of all those edges of $G$ whose endvertices both belong to $V(H)$. The order of $G$ is the cardinality $|V(G)|$, i.e. the number of vertices of $G$; the girth of $G$ is the smallest order of a cycle (if any) in $G$; a pendant vertex (or leaf) of $G$ is a vertex of degree $1$ in $G$.

A general question in extremal/structural graph theory~\cite{Bollobas1978,simonovits1983extremal,turan1941extremalaufgabe} is to find the minimum or maximum value of a prescribed graph parameter in a specified class of graphs. Tur{\'a}n's theorem~\cite{turan1941extremalaufgabe} dating back to 1941, characterises the $n$-vertex graphs with greatest number of edges that contain no complete graph as a subgraph; this is probably the most classical result in extremal graph theory. This question has been studied quite thoroughly for several other parameters including the popular invariant \emph{number of subtrees} of a tree (a connected graph with no cycle). Substantial work has been reported in the literature on the number of subtrees, see for example~\cite{GreedyWagner,graham1981trees,jamison1984monotonicity, jamison1983average,kirk2008largest,LiWang,szekely2007binary, szekely2005subtrees,yan2006enumeration}. In recent works~\cite{Audacegenral2018,AudaceGirth2018}, our main purpose was to extend some extremal results on the number of subtrees of a tree to more general classes of graphs such as connected graphs or unicylic graphs (connected graphs with only one cycle). In~\cite{Audacegenral2018}, order is prescribed for the class of all connected graphs and the class of all unicyclic graphs. Specifically, paper~\cite{Audacegenral2018} characterises those graphs (or unicyclic graphs) with $n$ vertices that minimise or maximise the number of connected induced subgraphs, thus extending some results of Sz{\'e}kely and Wang~\cite{szekely2005subtrees}. In~\cite{AudaceGirth2018}, further classes of graphs are considered, namely the class of all unicyclic graphs of order $n$ and with girth $g$, and the class of all unicyclic graphs of girth $g$, with $n$ vertices of which $p$ are pendant. For each of the aforementioned classes of graphs, the extreme numbers of connected induced subgraphs were found in~\cite{AudaceGirth2018}, and the extremal graph structures were also characterised. Extremal results on the total number of connected subgraphs (not necessary induced subgraphs) appeared recently in~\cite{pandey2018extremizing}. In general, there is no monotone relationship between the number of connected subgraphs and the number of connected induced subgraphs. In other words, if graph $G$ has more connected subgraphs than graph $H$, it is not necessary true that $G$ also contains more connected induced subgraphs than $H$.
  
In this note, we continue our systematic investigation on the number of connected induced subgraphs by considering two further classes of connected graphs. A component of $G$ is a maximal (with respect to the number of vertices) connected induced subgraph of $G$. By $G-u$, we mean the graph that results from deleting vertex $u$ and all edges incident with $u$ in $G$. A \emph{cut vertex} of $G$ is a vertex $u\in \V(G)$ with the property that $G-u$ has more components than $G$. In the present paper, which complements~\cite{Audacegenral2018,AudaceGirth2018}, we concentrate on two new classes of connected graphs for which we determine the extreme values and characterise the extremal graphs with respect to the number of connected induced subgraphs.  Section~\ref{connected c cut} deals with the class of all connected graphs of order $n$ with $c$ cut vertices, while in Section~\ref{connected p pendant} the focus is placed on the class of all connected graphs with $n$ vertices of which $p$ are pendant.

\medskip
The $n$-vertex path and the $n$-vertex star are denoted by $P_n$ and $S_n$, respectively. By $\mathbb{T}^1(n,p)$, we mean the tree obtained from the vertex disjoint graphs $S_{1+\lfloor p/2 \rfloor}$ and $S_{1+\lceil p/2 \rceil}$ by identifying their central vertices with the two leaves of $P_{n-p}$, respectively. Set $m:=\lfloor{(n-1)/p \rfloor},~l:=n-1-p\cdot m$ and denote by $\mathbb{T}^2(n,p)$ the rooted tree whose branches are $l$ copies of $P_{m+2}$ and $p-l$ copies of $P_{m+1}$. The extremal tree structures that minimise or maximise the number of subtrees of a tree with prescribed order and number of pendant vertices were characterised by Li and Wang~\cite{LiWang}, and Andriantiana et al.~\cite{GreedyWagner}, respectively. Li and Wang's result~\cite[Theorem~1]{LiWang} states that precisely the tree $\mathbb{T}^1(n,p)$ has the smallest number of subtrees, while Andriantiana et al.'s result~\cite[Corollary~4]{GreedyWagner} states that the maximum number of subtrees is achieved by the tree $\mathbb{T}^2(n,p)$. We shall prove (see Theorem~\ref{MinPpendGrap} in Section~\ref{connected p pendant}) that $\mathbb{T}^1(n,p)$ is the unique graph of order $n$ and with $p$ pendant vertices that minimises the number of connected induced subgraphs.

The Wiener index of a connected graph $G$ is defined as the sum of distances between all unordered pairs of vertices of $G$. The first results on this distance-based invariant date back to 1947 and are due to the chemist H.~Wiener~\cite{Wiener1947} who observed its strong correlation to the boiling point of certain chemical compounds. Subsequently, several authors have obtained sharp bounds on the Wiener index under various restrictions. A lower bound on the Wiener index, in terms of order and size, was given by Entringer et al.~\cite{Entringer1976}. An upper bound, depending on order, also appeared in~\cite{Entringer1976} by Entringer et al., and in ~\cite{Doyle} by Doyle and Graver. The maximum Wiener index among all cut vertex-free graphs was obtained by Plesn{\'\i}k~\cite{plesnik1984sum}. The Wiener index has been shown to correlate well with other chemical indices in applications~\cite{SzekelySemiR,WagnerCorr}. The tree $\mathbb{T}^1(n,p)$ was previously to Li and Wang's result~\cite[Theorem~1]{LiWang}, shown by Shi~\cite{Shi1993} to have the maximum Wiener index among all $n$-vertex trees with $p$ pendant vertices, while Entringer~\cite{entringer1999bounds}, and Entringer and Burns~\cite{burns1995graph} proved that $\mathbb{T}^2(n,p)$ is the tree of order $n$ with $p$ pendant vertices having the smallest Wiener index. The same is observed in our current context: for each of the graph classes in consideration, the graphs that are found to maximise the number of connected induced subgraphs were also recently reported in~\cite{plesnik1984sum,pandey2019wiener} to minimise the Wiener index, and vice versa.

\medskip
For a connected graph $G$, we denote by $n(G),~c(G),~p(G)$ (or simply $n,~c,~p$ if there is no danger of confusion) the order, number of cut vertices, and number of pendant vertices of $G$, respectively. It is well-known that if $G$ is a non-trivial connected graph (i.e. a graph of order at least two), then $c(G)\leq n(G)-2$ since a leaf of a spanning tree of $G$ cannot be a cut vertex of $G$. This bound is achieved by paths only (the cut vertices of a path are its vertices of degree $2$). From here onwards, we then assume that $n(G)>2$ and $c(G)< n(G)-2$. Clearly, if $T$ is a tree, then every vertex of $T$ is either a leaf or a cut vertex. Therefore, the identity $p(T)+c(T)=n(T)$ holds. Hence, the problem of finding the minimum (resp. maximum) number of connected induced subgraphs of an $n$-vertex tree having $c$ cut vertices is equivalent to the problem of finding the minimum (resp. maximum) number of connected induced subgraphs of an $n$-vertex tree having $n-c$ pendant vertices. However, as mentioned earlier, the extremal trees for the latter problem were already characterised by Li and Wang~\cite{LiWang}, and Andriantiana et al.~\cite{GreedyWagner}. This is a motivation for us to consider more general classes of connected graphs.

The complete graph of order $n$ and the cycle of order $n$ are denoted by $K_n$ and $C_n$, respectively. By $\deg_G(u)$, we mean the degree of vertex $u$ in the graph $G$. We denote by $\N(G)$ the number of connected induced subgraphs of $G$. By $\N(G)_u$, we mean those connected induced subgraphs of $G$ that contain vertex $u$, and $\N(G)_{u,v}$ stands for those connected induced subgraphs of $G$ that contain vertices $u$ and $v$. We simply write $G - u - v$ instead of $(G - u) - v$. 

We shall frequently employ the following three lemmas without further reference.
\begin{lemma}[\cite{szekely2005subtrees}]
We have $\N(P_n)=n(n+1)/2$ for all $n$. Moreover, if $u \in V(P_n)$, then $\N(P_n)_u\geq n$ with equality holding if and only if $u$ is a leaf.  
\end{lemma}

\begin{lemma}[\cite{Audacegenral2018}]
We have $\N(C_n)=n^2-n+1$ for all $n$. Moreover, if $u\in V(C_n)$, then we have $\N(C_n)_u=1+\binom{n}{2}$.
\end{lemma}

\begin{lemma}
We have $\N(K_n)=2^n -1$ for all $n$. Moreover, if $u\in V(K_n)$, then $\N(K_n)_u=2^{n-1}$ for all $n$.
\end{lemma}

\begin{proof}
Every induced subgraph of $K_n$ is a complete graph. Thus $\N(K_n)=2^n -1$. If $u\in V(K_n)$, then $\N(K_n)_u =\N(K_n) - \N(K_{n-1})=2^{n-1}$.
\end{proof}

Let $G$ be a connected graph. A \emph{block} of $G$ is a maximal (with respect to the number of vertices) cut vertex-free connected induced subgraph of $G$~\cite{harary1959elementary}. In particular, if $G$ is a non-trivial connected graph, then so are all blocks of $G$. Moreover, every block of $G$ is either $P_2$ or a cyclic graph since every tree of order three or more contains at least one cut vertex. As a first consequence of this definition, one deduces that the intersection of the vertex sets of any two distinct blocks of $G$ consists of at most one vertex~\cite{harary1963characterization}. 

\medskip
The proof techniques in this work build on several graph transformations, some of which are known to have a counterpart for the Wiener index. The rest of the paper is organised as follows: Section~\ref{connected c cut} contains extremal results on the number of connected induced subgraphs with $c$ cut vertices. Define $G(n_1;\ldots; n_q)$ to be the graph constructed as follows: we consider $q+1>3$ pairwise vertex disjoint graphs $K_q,P_{n_1},\ldots, P_{n_q}$ such that $V(K_q)=\{v_1,\ldots,v_q\}$; for every $j\in \{1,\ldots,q\}$, we let $u_j$ be a leaf of $P_{n_j}$ and identify $u_j$ with $v_j$. We prove (see Theorem~\ref{Main1cut}) that $G(s;\ldots;s; s+1;\ldots; s+1)$ ($n-c-t$ copies of $s$ followed by $t$ copies of $s+1$) is the unique connected graph of order $n$ and with $c$ cut vertices that has the greatest number of connected induced subgraphs. A formula in terms of $n$ and $c$ is also provided for $\N(G(s;\ldots;s; s+1;\ldots; s+1))$. We demonstrate in Theorem~\ref{Main2Cut} that the cycle $C_n$ has the smallest number of connected induced subgraphs among all cut vertex-free connected graphs of order $n$. The general case $c>0$ seems to be hard and we leave this as an open problem. Section~\ref{connected p pendant} considers the class of all connected graphs with $n$ vertices of which $p$ are pendant. The `maximal' graph in this case is already known; see~\cite{EricAudace}. We summarise this result in Theorem~\ref{maxn.p} and then prove its minimisation counterpart in Theorems~\ref{MinPpendGrap} and ~\ref{Theo:p=0}. Specifically, we show that for $p\neq 1$, the tree $\mathbb{T}^1(n,p)$ remains the unique graph of order $n$ and with $p$ pendant vertices that has the smallest number of connected induced subgraphs. For $p=0$ and $n>5$, we prove that the minimum number of connected induced subgraphs is realised by the so-called double tadpole graph, and that it is unique with this property. By the $n$-vertex double tadpole graph, we mean the graph constructed from the path of order $n-4$ and two vertex disjoint triangles by identifying bijectively the two leaves of the path with two other vertices, one from each triangle.

Our approach sometimes follows~\cite{pandey2019wiener}, adapted to our current setting. Throughout this note, all graphs are simply connected. We assume $n\geq 3$ and $p\leq n-2$ since the case $p=n-1\geq 2$ corresponds to the $n$-vertex star, i.e. a vertex and $n-1$ leaves attached to it. 

\section{Connected graphs with $c$ cut vertices}\label{connected c cut}

We define $\mathcal{H}(n,c)$ to be the set of all connected graphs with order $n$ and $c$ cut vertices.

\subsection{The maximisation problem}
In order to state the main result of this subsection, we need to go through some preparation. It is obvious that the complete graph $K_n$ uniquely realises the maximum number of connected induced subgraphs among all graphs in $\mathcal{H}(n,0)$.

Let $G$ be a non-trivial connected graph. The following properties about $G$ are elementary; see for instance~\cite{harary1963characterization,hararyPrins}.
\begin{enumerate}[(i)]
	\item Every cut vertex of $G$ belongs to at least two distinct blocks of $G$;
	\item Every two distinct blocks of $G$ have at most one vertex in common. Whenever they have a vertex in common, it must be a cut vertex of $G$.
	\item If $G$ has at least one cut vertex, then $G$ also has at least one block that contains exactly one cut vertex of $G$.
\end{enumerate}
We shall make frequent use of these properties without further reference. We begin with a series of important lemmas. The next two lemmas are straightforward. 
\begin{lemma}\label{AddEdgeBlock}
	If $G'$ is obtained from a non-trivial connected graph $G$ by adding an edge between two nonadjacent vertices of the same block of $G$, then $$c(G')=c(G).$$
\end{lemma}

Note that the above graph transformation (Lemma~\ref{AddEdgeBlock}) increases the number of edges in a block of $G$ while preserving the number of cut vertices of $G$. Our next transformation reduces the number of blocks of $G$ by one while preserving its number of cut vertices. 

\begin{lemma}\label{2block}
	Let	$B_1,B_2,B_3$ be three distinct blocks of a non-trivial connected graph $G$ such that $V(B_1)\cap V(B_2) \cap V(B_3)=\{w\}$. Assume that $G'$ is constructed from $G$ by adding an edge between a neighbour $v_1$ of $w$ in $B_1$ and a neighbour $v_2$ of $w$ in $B_2$. Then we have $$c(G)=c(G').$$
\end{lemma}

\begin{proof}
	Clearly, every cut vertex of $G'$ is a cut vertex of $G$ by construction. Let $z$ be a cut vertex of $G$. If $z \notin V(B_1)\cup V(B_2)\cup V(B_3)$, then all vertices in $V(B_1)\cup V(B_2)\cup V(B_3)$ are entirely contained in only one component of $G-z$. Thus $z$ is a cut vertex of $G'$. Otherwise, let $j\in \{1,2,3\}$ such that $z\in V(B_j)$. If $z\in \{v_1,v_2\}$, then $G-z$ and $G'-z$ are isomorphic graphs by definition of $G'$; otherwise $z\notin \{v_1,v_2\}$. If $z\neq w$, then $v_1$ and $v_2$ belong to the same component of $G-z$. Thus $z$ is a cut vertex of $G'$. Otherwise $z= w$ and so $V(B_1-z),V(B_2-z),V(B_3-z)$ are all contained entirely in distinct components of $G-z$. Since an edge is only added between $v_1$ and $v_2$ in $G$ to obtain $G'$, we deduce that the component of $G'-z$ that contains $B_3-z$ as a subgraph remains isolated in $G'-z$. Hence, $z$ is a cut vertex of $G'$.
\end{proof}

Consider $q+1>3$ pairwise vertex disjoint graphs $K_q,P_{n_1},\ldots, P_{n_q}$ such that $1\leq n_1\leq n_2-1$ and $V(K_q)=\{v_1,\ldots,v_q\}$. For every $j\in \{1,\ldots,q\}$, let $u_j$ be a leaf of $P_{n_j}$ and identify $u_j$ with $v_j$. We denote by $G(n_1;\ldots; n_q)$ the resulting graph. 

\begin{lemma}\label{PathOrder}
Let $H$ be a connected graph of order greater than two, and $u,v$ two distinct vertices of $H$ such that $\N(H)_{u,v}>1$ and $\N(H-u)_v \leq  \N(H-v)_u$. Let $H(n_1;n_2)$ be the graph obtained from $H$ by identifying $u$ with a leaf of $P_{n_1}$, and $v$ with a leaf of $P_{n_2}$ for some $1\leq n_1\leq n_2-1$. We have
\begin{align*}
\N(H(n_1;n_2)) \leq \N(H(n_1+1;n_2-1))\,.
\end{align*}
The inequality is strict if and only if $\N(H-u)_v < \N(H-v)_u$ or $n_1<n_2-1$. In particular, we get
\begin{align*}
	\N(G(n_1;n_2;\cdots; n_q)) \leq  \N(G(n_1+1;n_2-1;n_3;\cdots; n_q))
	\end{align*}
	if and only if $|n_1-n_2|\geq 1$. Equality holds if and only if $|n_1-n_2|= 1$.
\end{lemma}

\begin{proof}
We categorise subgraphs of $H(n_1;n_2)$ according to whether they contain an element of $\{u,v\}$ or not. Removing vertices $u$ and $v$ from $H(n_1;n_2)$ yields the graphs (possibly empty) $P_{n_1-1}, P_{n_2-1}$ and $H-u-v$. Thus $\N(P_{n_1-1})+ \N(P_{n_2-1})+\N(H-u-v)$ counts the number of connected induced subgraphs of $H(n_1;n_2)$ that contain none of the vertices $u,v$. On the other hand, $n_1\cdot \N(H-v)_u + n_2\cdot \N(H-u)_v$ counts the number of connected induced subgraphs of $H(n_1;n_2)$ that contain $u$ or $v$ but not both. The number of connected induced subgraphs of $H(n_1;n_2)$ that contain both $u$ and $v$ is given by $n_1\cdot n_2\cdot \N(H)_{u,v}$. Hence, we get
\begin{align*}
	\N(H(n_1;n_2))&=n_1\cdot n_2\cdot \N(H)_{u,v}+n_1\cdot \N(H-v)_u + n_2\cdot \N(H-u)_v\\
	&+\N(P_{n_1-1})+ \N(P_{n_2-1})+\N(H-u-v)\,.
\end{align*}
This implies that
\begin{align*}
\N(H(n_1;n_2))-&\N(H(n_1+1;n_2-1))=(n_1\cdot n_2 -(n_1+1)(n_2-1))\N(H)_{u,v}\\
& + (n_1-(n_1+1)) \N(H-v)_u +(n_2- (n_2-1)) \N(H-u)_v\\
&+\binom{n_1}{2} - \binom{n_1+1}{2} + \binom{n_2}{2} - \binom{n_2-1}{2}\\
&=(n_1-n_2+1)(\N(H)_{u,v} -1) +  \N(H-u)_v -  \N(H-v)_u \leq 0\,.
\end{align*}
Moreover, this inequality becomes an equality if and only if $\N(H-u)_v = \N(H-v)_u$ and $n_1=n_2-1$. This proves the lemma.
\end{proof}

\begin{lemma}\label{oneCut}
Let $H(n;l)$ be the graph constructed from the two vertex disjoint complete graphs $K_l$ and $K_{n+1-l}$ by identifying $u \in V(K_l)$ with $v \in V(K_{n+1-l})$ for some $n\geq 3$ and $2\leq l\leq (n+1)/2$. Then we have 
\begin{align*}
\N(H(n;2))>\N(H(n;3))>\cdots >\N(H(n;\lfloor (n+1)/2 \rfloor))\,.
\end{align*}
\end{lemma}

\begin{proof}
We have
\begin{align*}
\N(H(n;l))&=\N(H(n;l))_u+\N(H(n;l)-u)\\
&=\N(K_l)_u\cdot \N(K_{n+1-l})_v + \N(K_{l-1})+\N(K_{n-l}) =2^{n-1}+2^{l-1}+2^{n-l}-2
\end{align*}
which implies that 
\begin{align*}
\N(H(n;l))-\N(H(n;l+1))=2^{n-l-1} - 2^{l-1}>0
\end{align*}
for all $2\leq l \leq (n-1)/2$. The statement of the lemma follows.
\end{proof}

\begin{lemma}\label{Block.two.cutPend}
Let $l\geq 3,~r\geq 2$ be two positive integers and $K_l,K_r$ two vertex disjoint complete graphs such that $V(K_r)=\{w_1,\ldots,w_r\}$. Consider $r-1$ vertex disjoint connected graphs $R_2,\ldots,R_r$ such that $v_j \in V(R_j)$ for every $j \in \{2,\ldots,r\}$ and $|V(R_2)|>1$. Identify $w_1$ with a fixed vertex $u\in V(K_l)$. Further, identify $w_j$ with $v_j$ for all $j>1$. Call the resulting graph $G_1$. Fix $v \neq u \in V(K_l)$ and let $G_2$ be constructed from $G_1$ by deleting the edges joining $v$ to a neighbour, except $u$, of $v$ in $K_l$. Finally, let $G_3$ be constructed from $G_2$ by making the graph induced by $V(K_l - u -v) \cup V(K_r-w_1)$ in $G_2$ a complete graph. We have
\begin{align*}
\N(G_3)> \N(G_1) \quad \text{and} \quad c(G_1)=c(G_3)\,.
\end{align*}

\end{lemma}

\begin{proof}
Denote by $H$ the graph induced by $V(G_1)-V(K_l -u)$ in $G_1$. We have
\begin{align*}
\N(G_1)&=\N(G_1)_u+\N(G_1 -u)=\N(K_l)_u\cdot \N(H)_{w_1}+\N(K_{l-1})+\N(H-w_1)\\
\N(G_2)&=\N(G_2)_u+\N(G_2 -u)= 2 \N(H)_{w_1}\cdot \N(K_l -v)_u +1 +\N(K_{l-2})+\N(H-w_1)\,.
\end{align*}
In particular, we get $\N(G_1)-\N(G_2)=2^{l-2}-1$. Clearly, $G_2$ is a subgraph of $G_3$ by construction. Let $ S_1 \subseteq V(K_l -u -v)$ and $ S_2 \subseteq V(K_r-w_1)$ be two nonempty subsets of vertices of $V(G_2)=V(G_3)$. These choices of $S_1$ and $S_2$ are possible since $l\geq 3$ and $r\geq 2$. The graph induced by $S_1 \cup S_2$ in $G_2$ is disconnected while the graph induced by $S_1 \cup S_2$ in $G_3$ is connected. The total number of these connected induced subgraphs in $G_3$ is given by $(2^{l-2}-1)(2^{r-1}-1)$. Let $z$ be a vertex adjacent to $v_2$ in $R_2$. Vertex $z$ exists since $|V(R_2)|>1$. The graph induced by $S_1 \cup \{v_2,z\}$ in $G_3$ is connected and different from all the subgraphs induced by $S_1 \cup S_2$ as $z\notin S_1 \cup S_2$. Moreover, $S_1 \cup \{v_2,z\}$ induces a disconnected graph in $G_2$. The total number of such connected induced subgraphs in $G_3$ is $2^{l-2}-1$. Therefore, we deduce that
\begin{align*}
\N(G_3)-\N(G_2) \geq (2^{l-2}-1)2^{r-1} \geq 2(2^{l-2}-1)\,.
\end{align*}				
It follows that $\N(G_3)-\N(G_1) \geq 2^{l-2}-1 > 0$.

Now since $v$ is a leaf of $G_3$ and $u$ is adjacent to $v$ in $G_3$, we conclude that $u$ remains a cut-vertex of $G_3$ while $v$ remains a non cut-vertex of $G_3$. Moreover, all other vertices of $G_1$ preserve their status (cut vertex or not) in $G_3$. This proves that $c(G_1)=c(G_3)$, completing the proof.
\end{proof}

Next, we describe another graph transformation that will also be useful for our analysis. It is a result that is similar in nature to but different from Lemma~\ref{Block.two.cutPend}. It does, however, complement Lemma~\ref{Block.two.cutPend}.

\begin{lemma}\label{Block.Two.cutCenter}
Let $K_l,K_r$ be two complete graphs with (disjoint) vertex sets
\begin{align*}
V(K_l)=\{u_1,\ldots,u_l\},~V(K_r)=\{w_1,\ldots,w_r\}
\end{align*}
for some $l,r\geq 3$. Consider $l+r-1$ vertex disjoint connected graphs $M,L_2,\ldots,L_l,R_2,\ldots,R_r$ such that $x_j \in V(L_j)$ and $z_j \in V(R_j)$ for all $j \neq 1$. Let $v_1,v_2$ be two distinct vertices of $M$. Identify $u_1$ with $v_1$, and $w_1$ with $v_2$. Further, identify $u_j$ with $x_j$, and $w_j$ with $z_j$ for all $j\neq 1$. Denote by $G_1$ the resulting graph. Let $G_2$ be obtained from $G_1$ by removing the edges joining $u_1$ to a neighbour, except $u_2$, of $u_1$ in $K_l$; see Figure~\ref{PictG1G2}. Let $w'$ be a fixed neighbour of $v_1$ in $M$ such that $w'$ lies on a shortest $v_1-v_2$ path $P$ in $G_2$. A new graph $G_3$ is constructed from $G_2$ by adding an edge between $w'$ and all vertices $u_3,\ldots,u_l$. We have
\begin{align*}
c(G_1)=c(G_3)\,.
\end{align*}

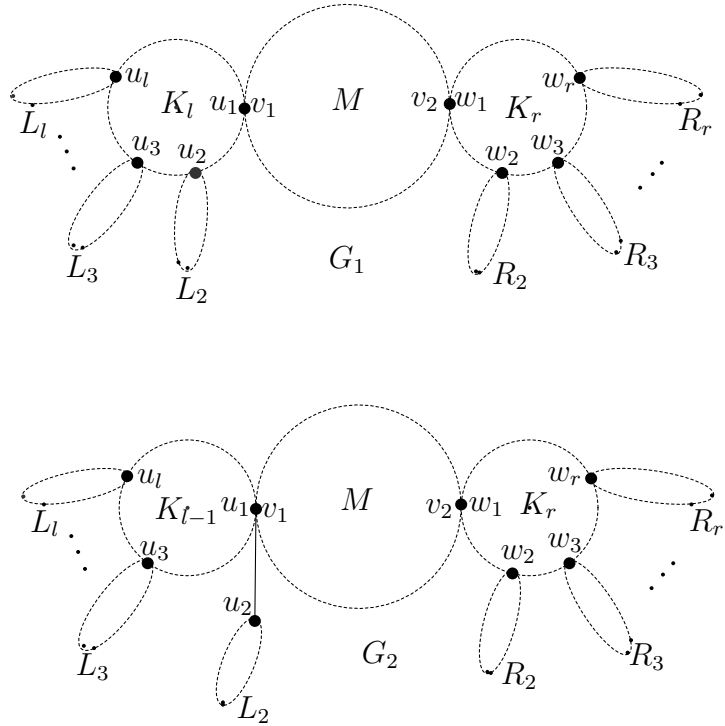
\begin{figure}[!h]
	\definecolor{tttttt}{rgb}{0.2,0.2,0.2}
	\begin{tikzpicture}[scale=1.5,line cap=round,line join=round,>=triangle 45,x=1.0cm,y=1.0cm]
	\draw [dash pattern=on 1pt off 1pt] (2.09,2.06) circle (0.6cm);
	\draw [dash pattern=on 1pt off 1pt] (5.09,2.06) circle (0.6cm);
	\draw [shift={(3.59,2.07)},dash pattern=on 1pt off 1pt]  plot[domain=0.02:3.16,variable=\t]({1*0.9*cos(\t r)+0*0.9*sin(\t r)},{0*0.9*cos(\t r)+1*0.9*sin(\t r)});
	\draw [shift={(3.59,2.07)},dash pattern=on 1pt off 1pt]  plot[domain=-3.12:0.02,variable=\t]({1*0.9*cos(\t r)+0*0.9*sin(\t r)},{0*0.9*cos(\t r)+1*0.9*sin(\t r)});
	\draw [rotate around={52.44:(1.46,1.2)},dash pattern=on 1pt off 1pt] (1.46,1.2) ellipse (0.49cm and 0.16cm);
	\draw [rotate around={11.17:(1.11,2.25)},dash pattern=on 1pt off 1pt] (1.11,2.25) ellipse (0.48cm and 0.13cm);
	\draw [rotate around={75.35:(4.83,1.05)},dash pattern=on 1pt off 1pt] (4.83,1.05) ellipse (0.47cm and 0.14cm);
	\draw [rotate around={-57.01:(5.69,1.18)},dash pattern=on 1pt off 1pt] (5.69,1.18) ellipse (0.49cm and 0.14cm);
	\draw [rotate around={-8.09:(6.16,2.25)},dash pattern=on 1pt off 1pt] (6.16,2.25) ellipse (0.55cm and 0.14cm);
	\draw (3.33,2.33) node[anchor=north west] {$M$};
	\draw (4.90,2.3) node[anchor=north west] {$ K_r $};
	\draw (2.28,2.25) node[anchor=north west] {$ u_1 $};
	\draw (2.64,2.2) node[anchor=north west] {$v_1 $};
	\draw (4.08,2.23) node[anchor=north west] {$ v_2 $};
	\draw (4.46,2.25) node[anchor=north west] {$ w_1 $};
	\draw (2.42,0.49) node[anchor=north west] {$ L_2 $};
	\draw (1.01,0.87) node[anchor=north west] {$ L_3 $};
	\draw (0.61,2.14) node[anchor=north west] {$ L_l $};
	\draw (4.74,0.79) node[anchor=north west] {$R_2 $};
	\draw (5.86,0.96) node[anchor=north west] {$ R_3 $};
	\draw (6.39,2.15) node[anchor=north west] {$ R_r $};
	\draw (2.3,1.37) node[anchor=north west] {$ u_2 $};
	\draw (1.57,1.86) node[anchor=north west] {$ u_3 $};
	\draw (1.55,2.5) node[anchor=north west] {$ u_l $};
	\draw (4.74,1.83) node[anchor=north west] {$ w_2 $};
	\draw (5.15,1.93) node[anchor=north west] {$ w_3$};
	\draw (5.15,2.52) node[anchor=north west] {$ w_r $};
	\draw (3.52,0.98) node[anchor=north west] {$ G_2$};
	\draw (2.68,1.06)-- (2.69,2.05);
	\draw [rotate around={68.56:(2.54,0.7)},dash pattern=on 1pt off 1pt] (2.54,0.7) ellipse (0.41cm and 0.14cm);
	\draw (1.7,2.25) node[anchor=north west] {$K_{l-1} $};
	\draw [dash pattern=on 1pt off 1pt] (1.99,5.6) circle (0.6cm);
	\draw [dash pattern=on 1pt off 1pt] (4.99,5.6) circle (0.6cm);
	\draw [shift={(3.49,5.61)},dash pattern=on 1pt off 1pt]  plot[domain=0.02:3.16,variable=\t]({1*0.9*cos(\t r)+0*0.9*sin(\t r)},{0*0.9*cos(\t r)+1*0.9*sin(\t r)});
	\draw [shift={(3.49,5.61)},dash pattern=on 1pt off 1pt]  plot[domain=-3.12:0.02,variable=\t]({1*0.9*cos(\t r)+0*0.9*sin(\t r)},{0*0.9*cos(\t r)+1*0.9*sin(\t r)});
	\draw [rotate around={85.09:(2.12,4.6)},dash pattern=on 1pt off 1pt] (2.12,4.6) ellipse (0.45cm and 0.14cm);
	\draw [rotate around={52.44:(1.37,4.74)},dash pattern=on 1pt off 1pt] (1.37,4.74) ellipse (0.49cm and 0.16cm);
	\draw [rotate around={11.17:(1.01,5.79)},dash pattern=on 1pt off 1pt] (1.01,5.79) ellipse (0.48cm and 0.13cm);
	\draw [rotate around={75.35:(4.73,4.58)},dash pattern=on 1pt off 1pt] (4.73,4.58) ellipse (0.47cm and 0.14cm);
	\draw [rotate around={-57.01:(5.6,4.72)},dash pattern=on 1pt off 1pt] (5.6,4.72) ellipse (0.49cm and 0.14cm);
	\draw [rotate around={-8.09:(6.06,5.79)},dash pattern=on 1pt off 1pt] (6.06,5.79) ellipse (0.55cm and 0.14cm);
	\draw (1.75,5.84) node[anchor=north west] {$K_l$};
	\draw (3.25,5.87) node[anchor=north west] {$M$};
	\draw (4.77,5.8) node[anchor=north west] {$ K_r $};
	\draw (2.18,5.8) node[anchor=north west] {$ u_1 $};
	\draw (2.54,5.78) node[anchor=north west] {$ v_1 $};
	\draw (3.95,5.83) node[anchor=north west] {$ v_2 $};
	\draw (4.33,5.83) node[anchor=north west] {$ w_1 $};
	\draw (1.88,4.2) node[anchor=north west] {$ L_2 $};
	\draw (0.92,4.39) node[anchor=north west] {$ L_3 $};
	\draw (0.51,5.66) node[anchor=north west] {$ L_l $};
	\draw (4.66,4.33) node[anchor=north west] {$ R_2 $};
	\draw (5.81,4.5) node[anchor=north west] {$ R_3 $};
	\draw (6.29,5.69) node[anchor=north west] {$ R_r $};
	\draw (1.9,5.37) node[anchor=north west] {$ u_2 $};
	\draw (1.5,5.44) node[anchor=north west] {$ u_3 $};
	\draw (1.44,6.04) node[anchor=north west] {$ u_l $};
	\draw (4.6,5.35) node[anchor=north west] {$ w_2 $};
	\draw (5.,5.45) node[anchor=north west] {$ w_3$};
	\draw (5.1,6.) node[anchor=north west] {$ w_r $};
	\draw (3.23,4.47) node[anchor=north west] {$ G_1 $};
	\fill [color=tttttt] (2.09,2.06) circle (0.5pt);
	\fill [color=black] (5.09,2.06) circle (0.5pt);
	\fill [color=black] (2.69,2.05) circle (1.5pt);
	\fill [color=black] (4.49,2.09) circle (1.5pt);
	\fill [color=black] (1.74,1.57) circle (1.5pt);
	\fill [color=black] (1.56,2.34) circle (1.5pt);
	\fill [color=black] (4.94,1.48) circle (1.5pt);
	\fill [color=black] (5.44,1.57) circle (1.5pt);
	\fill [color=black] (5.63,2.32) circle (1.5pt);
	\fill [color=black] (1.18,0.84) circle (0.5pt);
	\fill [color=black] (1.27,0.82) circle (0.5pt);
	\fill [color=tttttt] (0.65,2.16) circle (0.5pt);
	\fill [color=black] (0.83,2.09) circle (0.5pt);
	\fill [color=black] (4.72,0.61) circle (0.5pt);
	\fill [color=black] (4.75,0.6) circle (0.5pt);
	\fill [color=black] (5.95,0.78) circle (0.5pt);
	\fill [color=black] (5.98,0.89) circle (0.5pt);
	\fill [color=black] (6.69,2.17) circle (0.5pt);
	\fill [color=black] (6.51,2.09) circle (0.5pt);
	\fill [color=black] (1.07,1.81) circle (0.5pt);
	\fill [color=black] (1.13,1.67) circle (0.5pt);
	\fill [color=black] (1.19,1.52) circle (0.5pt);
	\fill [color=black] (6.35,1.57) circle (0.5pt);
	\fill [color=black] (6.26,1.47) circle (0.5pt);
	\fill [color=black] (6.16,1.35) circle (0.5pt);
	\fill [color=black] (2.68,1.06) circle (1.5pt);
	\fill [color=black] (2.4,0.34) circle (0.5pt);
	\fill [color=black] (2.4,0.32) circle (0.5pt);
	\fill [color=tttttt] (1.99,5.6) circle (0.5pt);
	\fill [color=black] (4.99,5.6) circle (0.5pt);
	\fill [color=black] (2.59,5.59) circle (1.5pt);
	\fill [color=black] (4.39,5.63) circle (1.5pt);
	\fill [color=tttttt] (2.16,5.02) circle (1.5pt);
	\fill [color=black] (1.65,5.11) circle (1.5pt);
	\fill [color=black] (1.46,5.87) circle (1.5pt);
	\fill [color=black] (4.85,5.02) circle (1.5pt);
	\fill [color=black] (5.34,5.11) circle (1.5pt);
	\fill [color=black] (5.53,5.86) circle (1.5pt);
	\fill [color=black] (2.09,4.18) circle (0.5pt);
	\fill [color=black] (2.01,4.23) circle (0.5pt);
	\fill [color=black] (1.09,4.38) circle (0.5pt);
	\fill [color=black] (1.17,4.36) circle (0.5pt);
	\fill [color=tttttt] (0.56,5.7) circle (0.5pt);
	\fill [color=black] (0.73,5.62) circle (0.5pt);
	\fill [color=black] (4.62,4.15) circle (0.5pt);
	\fill [color=black] (4.65,4.14) circle (0.5pt);
	\fill [color=black] (5.85,4.32) circle (0.5pt);
	\fill [color=black] (5.89,4.42) circle (0.5pt);
	\fill [color=black] (6.59,5.71) circle (0.5pt);
	\fill [color=black] (6.41,5.63) circle (0.5pt);
	\fill [color=black] (0.97,5.34) circle (0.5pt);
	\fill [color=black] (1.03,5.21) circle (0.5pt);
	\fill [color=black] (1.09,5.06) circle (0.5pt);
	\fill [color=black] (6.25,5.11) circle (0.5pt);
	\fill [color=black] (6.17,5.01) circle (0.5pt);
	\fill [color=black] (6.06,4.89) circle (0.5pt);
	\end{tikzpicture}
	\caption{The graphs $G_1$ and $G_2$ constructed in Lemma~\ref{Block.Two.cutCenter}.}\label{PictG1G2}
\end{figure}

Furthermore, let $L$ be the graph induced by $\{u_1\}\cup V(L_2)\cup \cdots \cup V(L_l)$ in $G_1$, and $R$ the graph induced by $\{w_1\}\cup V(R_2)\cup \cdots \cup V(R_r)$ in $G_1$. Assume that $\N(R)_{w_1}\geq \N(L)_{u_1}$. Then we have
\begin{align*}
N(G_3)>N(G_1)\,.
\end{align*}
\end{lemma}	

\begin{proof}
It is clear by construction that $u_1\in V(G_1)$ remains a cut vertex of $G_2$. This is because $u_1$ is adjacent to $u_2$, and $u_2$ is adjacent to no vertex of $G_2$ outside $V(L_2)\cup \{u_1\}$ in $G_2$. Thus, all cut (resp. non cut) vertices of $G_1$ remain cut (resp. non cut) vertices of $G_2$. Therefore, we have $c(G_1)=c(G_2)$. On the other hand, since edges are only added between $w'$ and the vertices $u_3,\ldots,u_l$ in $G_2$ to obtain $G_3$, it is clear that the following hold:
\begin{itemize}
\item All non cut vertices of $G_2$ remain non cut vertices of $G_3$;
\item All cut vertices of $G_2$ that do not belong to $V(M)- \{v_1,v_2\}$ remain cut vertices of $G_3$.
\end{itemize}
Let $\theta$ be a cut vertex of $G_2$ such that $\theta \in V(M)-\{v_1,v_2\}$. We show that $\theta$ is also a cut vertex of $G_3$. If $\theta=w'$, then $G_3-\theta$ and $G_2-\theta$ are isomorphic graphs. So assume that $\theta\neq w'$. Then $w'$ must belong to the component, say $C$ of $G_2-\theta$ that contains $v_1$, since otherwise, every $v_1-w'$ path must pass through $\theta$. In particular, we get $\theta \in \{v_1,w'\}$ as $v_1w'$ is an edge of $G_2$: this is a contradiction to the choice of $\theta$. Hence, $w' \in V(C)$.

Note that $C$ also contains all of $u_3,\ldots,u_l$ since $\theta \in V(M)-\{v_1,v_2\}$ and $C$ is a component of $G_2-\theta$ that contains $u_1(=v_1)$. Since $w' \in V(C)$, we then deduce that all other (different from $C$) components of $G_2-\theta$ remain components of $G_3-\theta$. This proves that $\theta$ is indeed a cut vertex of $G_3$. In particular, we get $c(G_3)=c(G_2)=c(G_1)$.

\medskip
Let $x_1 \in \{u_1,v_1\}$ and denote by $L_1$ the graph induced by $V(M)\cup V(R)$ in $G_1$. Then the vertex set of $G_1$ can be partitioned into $V(L_1),\ldots, V(L_l)$. Thus, for a subset $S\subseteq V(G_1)$ containing a vertex of $V(L_i)$ and a vertex of $V(L_j)$, where $i\neq j$ to induce a connected graph in $G_1$, it is necessary to have $x_i,x_j \in S$. Therefore, we get
\begin{align*}
\N(G_1)=\sum_{j=1}^l \N(L_j -x_j) + \prod_{j=1}^l (1+\N(L_j)_{x_j})-1
\end{align*}
as a formula for the number of connected induced subgraphs of $G_1$. Likewise, denote by $L'$ the graph induced by $V(L_1)\cup V(L_2)$ in $G_2$. The set $V(G_2)$ can also be partitioned into $V(L'), V(L_3),\ldots, V(L_l)$. Thus, we get
\begin{align*}
\N(G_2)=\N(L' -x_1)+ \sum_{j=3}^l \N(L_j -x_j) + (1+\N(L')_{x_1})\prod_{j=3}^l (1+\N(L_j)_{x_j})-1
\end{align*}
in the same way as for $G_1$. On the other hand, we have
\begin{align*}
\N(L' -x_1)=
\N(L_2)+\N(L_1 -x_1)~\text{and}~\N(L')_{x_1}=\N(L_1)_{x_1}(1+\N(L_2)_{x_2})\,.
\end{align*}
Therefore, we obtain
\begin{align*}
\N(G_1)-\N(G_2)=\N(L_2)_{x_2} \Big(\prod_{j=3}^l (1+\N(L_j)_{x_j}) -1\Big)
\end{align*}
after simplification. By construction, $G_3$ contains $G_2$ as a subgraph. We now find a lower bound on $\N(G_3)-\N(G_2)$ by solely counting certain subsets of $V(G_3)=V(G_2)$ that induce a connected graph in $G_3$ and a disconnected graph in $G_2$. Let $S_1\neq \emptyset$ be a subset of $V(L_3)\cup \cdots \cup V(L_l)$ such that $S_1$ contains $x_j$ whenever $S_1$ contains an element of $V(L_j)$. Recall that $P$ is a fixed shortest $v_1-v_2$ path in $G_2$ that contains $w'$. Denote by $R'$ the graph induced by $V(R)\cup V(P-{v_1})$ in $G_2$. Let $S_2 \neq \emptyset$ a subset of $V(R')$ that contains $w'$. Since $w'\neq v_1$ is adjacent to all of $x_3,\ldots,x_l$ in $G_3$, we deduce that $S_1\cup S_2$ always induces a connected graph in $G_3$. However, the graph induced by $S_1\cup S_2$ in $G_2$ is always disconnected as there is no edge from an element of $S_1$ to an element of $S_2$ in $G_2$. Therefore, we obtain a total of 
\begin{align*}
\N(R')_{w'} \Big(\prod_{j=3}^l (1+\N(L_j)_{x_j}) -1\Big)
\end{align*}
such sets $S_1\cup S_2$ inducing a connected graph in $G_3$ and a disconnected graph in $G_2$. Since $\N(L)_{u_1}> 1+\N(L_2)_{u_2}$, we use the trivial inequality $\N(R')_{w'}\geq \N(R)_{w_1}$ alongside the assumption $\N(R)_{w_1}\geq \N(L)_{u_1}$ to derive that
\begin{align*}
\N(G_3)-\N(G_2)\geq \N(R)_{w_1} \Big(\prod_{j=3}^l (1+\N(L_j)_{x_j}) -1\Big) >(1+\N(L_2)_{u_2}) \Big(\prod_{j=3}^l (1+\N(L_j)_{x_j}) -1\Big)\,.
\end{align*}
This implies that
\begin{align*}
\N(G_3)-\N(G_1)> \prod_{j=3}^l (1+\N(L_j)_{x_j}) -1 \geq  \N(L_3)_{x_3}> 0\,,
\end{align*}
completing the proof.
\end{proof}

It is required in Lemma~\ref{Block.Two.cutCenter} that $|V(M)|>1$. Lemma~\ref{Block.Two.cutCenterM=1} below covers the special case where  $|V(M)|=1$.

\begin{lemma}\label{Block.Two.cutCenterM=1}
Let $K_l,K_r$ be two complete graphs with (disjoint) vertex sets
	\begin{align*}
	V(K_l)=\{u_1,\ldots,u_l\},~V(K_r)=\{w_1,\ldots,w_r\}
	\end{align*}
	for some $l,r\geq 3$. Consider $l+r-2$ vertex disjoint connected graphs $L_2,\ldots,L_l, R_2,\ldots,R_r$ such that $x_j \in V(L_j)$ and $z_j \in V(R_j)$ for all $j \in \{2,\ldots,r\}$. Identify $u_1$ with $w_1$, $u_j$ with $x_j$, and $w_j$ with $z_j$ for all $j\neq 1$. Denote by $G_1$ the resulting graph. Let $G_2$ be obtained from $G_1$ by removing the edges joining $u_1$ to a neighbour, except $u_2$, of $u_1$ in $K_l$. Let $G_3$ be constructed from $G_2$ by making the graph induced by the set $V(K_l-u_1-u_2)\cup V(K_r-w_1)$ a complete graph. We have $c(G_1)=c(G_3)$. Furthermore, assume that $\N(R_2)_{z_2} \geq \N(L_2)_{x_2}$. Then we have
	\begin{align*}
	N(G_3)>N(G_1)\,.
	\end{align*}
\end{lemma}	

\begin{proof}
The proof is done in analogy to Lemma~\ref{Block.Two.cutCenter} with the following simple modification. Let $R$ be the graph induced by $\{w_1\}\cup V(R_2)\cup \cdots \cup V(R_r)$ in $G_1$. Denote by $R'$ the graph induced by $V(R)\cup V(L_2)$ in $G_2$. We have
	\begin{align*}
	\N(G_1)&=\N(R -w_1)+\sum_{j=2}^l \N(L_j -x_j) + (1+\N(R)_{w_1})\prod_{j=2}^l (1+\N(L_j)_{x_j})-1\,,\\
	\N(G_2)&=\N(R' -w_1)+\sum_{j=3}^l \N(L_j -x_j) + (1+\N(R')_{w_1})\prod_{j=3}^l (1+\N(L_j)_{x_j})-1\,,
	\end{align*}
and
	\begin{align*}
	\N(R' -w_1)=\N(L_2)+\N(R -w_1),~~\N(R')_{w_1}=\N(R)_{w_1}(1+\N(L_2)_{x_2})\,.
	\end{align*}
It follows that
	\begin{align*}
	\N(G_1)-\N(G_2)=\N(L_2)_{x_2} \Big(\prod_{j=3}^l (1+\N(L_j)_{x_j}) -1\Big)\,.
	\end{align*}
Clearly, every subgraph of $G_2$ is also a subgraph of $G_3$. Let $S_1\neq \emptyset$ be a subset of $V(L_3)\cup \cdots \cup V(L_l)$ such that $S_1$ contains $x_j$ whenever $S_1$ contains an element of $V(L_j)$. Likewise, let $S_2 \neq \emptyset$ be a subset of $V(R-w_1)$ such that $S_2$ contains $z_j$ whenever $S_2$ contains an element of $V(R_j)$. The set $S_1\cup S_2$ always induces a disconnected graph in $G_2$, and a connected graph in $G_3$. Therefore, we get
	\begin{align*}
		\N(G_3)-\N(G_2)\geq \Big(\prod_{j=2}^r (1+\N(R_j)_{z_j}) -1\Big) \Big(\prod_{j=3}^l (1+\N(L_j)_{x_j}) -1\Big)\,.
	\end{align*}
On the other hand, we have
\begin{align*}
\prod_{j=2}^r (1+\N(R_j)_{z_j}) -1 \geq (1+\N(R_2)_{z_2})(1+\N(R_3)_{z_3}) -1\geq 1+ 2\N(R_2)_{z_2}\,.
\end{align*}	
Hence, using the assumption	$\N(R_2)_{z_2} \geq \N(L_2)_{x_2}$, we deduce that
\begin{align*}
\N(G_3)-\N(G_2)\geq (1+ 2\N(L_2)_{x_2}) \Big(\prod_{j=3}^l (1+\N(L_j)_{x_j}) -1\Big)\,,
\end{align*}
which implies that $\N(G_3)-\N(G_1)>0$. This completes the proof of the lemma.
\end{proof}

We are now ready to formulate a characterisation of all graphs maximising the number of connected induced subgraphs in the set $\mathcal{H}(n,c)$. At this point, it can be recalled that $G(n_1;\ldots; n_q)$ is the graph constructed as follows: we consider $q+1>3$ pairwise vertex disjoint graphs $K_q,P_{n_1},\ldots, P_{n_q}$ such that $n_1\leq n_2,~n_2>1$ and $V(K_q)=\{v_1,\ldots,v_q\}$. For every $j\in \{1,\ldots,q\}$, we let $u_j$ be a leaf of $P_{n_j}$ and identify $u_j$ with $v_j$. 

\begin{theorem}\label{Main1cut}
Let $n>1$ and $0\leq c\leq n-2$. Denote by $t$ the residue of $n$ modulo $n-c$, and set $s=\lfloor n/(n-c) \rfloor$. We have
\begin{align*}
\N(H)\leq (n-c-t)\binom{s}{2}+t\binom{s+1}{2}+
(s+1)^{n-c-t}(s+2)^t -1
\end{align*}
for all $H \in \mathcal{H}(n,c)$. Equality holds if and only if $H$ is isomorphic to the graph $G(s;\ldots;s; s+1;\ldots; s+1)$ ($n-c-t$ copies of $s$ followed by $t$ copies of $s+1$).
\end{theorem}

\begin{proof}
First off, note that if $B$ is a block of a non-trivial connected graph $G$, then $B$ is necessarily `surrounded' by  $|V(B)|$ (possibly trivial) connected induced subgraphs of $G$ whose vertex sets are pairwise disjoint. In other words, the removal of all edges of $B$ in $G$ must leave $|V(B)|$ connected graphs.

Let $\mathbb{H} \in \mathcal{H}(n,c)$ be a graph with order $n$ and $c$ cut vertices that maximises the number of connected induced subgraphs. We know, by repeatedly applying Lemma~\ref{AddEdgeBlock}, that all blocks of $\mathbb{H}$ are non-trivial complete graphs. We are going to prove that all blocks of $\mathbb{H}$, except possibly only one, are in fact of order $2$. The statement is obvious for $c=0$ since $\mathbb{H}=K_n$ in this case. So we assume that $c\geq 1$. By repeatedly invoking Lemma~\ref{2block}, we can further assume that every cut vertex of $G$ belongs to precisely two distinct blocks of $\mathbb{H}$. If $c=1$, then $\mathbb{H}$ has precisely two blocks, say $K_l$ and $K_{n+1-l}$ for some $2\leq l\leq (n+1)/2$. Thus, in this case, the statement holds true by Lemma~\ref{oneCut}. So we assume that $c\geq 2$. Consider a block $K_l$ of $\mathbb{H}$ such that $l\geq 3$. We consider two separate cases depending on whether $K_l$ contains one or more cut vertices of $\mathbb{H}$.

Assume that $K_l$ contains precisely one cut vertex, say $w_1$ of $\mathbb{H}$. Let $w_2 \neq w_1$ be another cut vertex of $\mathbb{H}$ such that both $w_1$ and $w_2$ belong the the same block $K_r$ of $\mathbb{H}$. Thus $w_2$ also belongs to a further block $B$ of $\mathbb{H}$ different from $K_r$. This kind of description for $\mathbb{H}$ yields exactly the graph $G_1$ constructed in Lemma~\ref{Block.two.cutPend}, where the graph $R_2$ in Lemma~\ref{Block.two.cutPend} contains $B$ as a subgraph and $w_2 \in V(R_2)$. Note that the graph transformation described in Lemma~\ref{Block.two.cutPend} preserves the number of cut vertices when passing from $G_1$ to $G_3$ but creates a new block of order $2$ in $G_3$. It is shown in Lemma~\ref{Block.two.cutPend} that $\N(G_3)>\N(G_1)=\N(\mathbb{H})$. However, this is impossible from the choice of $\mathbb{H}$. Hence, we must have $l=2$.

Assume that $K_l$ contains two or more cut vertices, say $u_1,u_2$ of $\mathbb{H}$. If there is no other block that contains two or more cut vertices of $\mathbb{H}$, then we are done immediately by Lemma~\ref{Block.two.cutPend}. This is because Lemma~\ref{Block.two.cutPend} states that in a `maximal' graph, all blocks that contain only one cut vertex must be of order $2$.

Otherwise, let $K_r \neq K_l$ be another block containing two or more cut vertices, say $w_1,w_2$ of $\mathbb{H}$. We can assume that $r\geq 3$ since otherwise, there is nothing more to prove.
We observe two possible situations:
\begin{enumerate}[{Case} 1:]
\item $V(K_l)\cap V(K_r)=\emptyset$. In this case, there exists a non-trivial connected graph $M$ that contains both $u_1,w_1$ and no other vertex of $V(K_l)\cup V(K_r)$. In particular, $\mathbb{H}$ can be described in the same way as the graph $G_1$ defined in Lemma~\ref{Block.Two.cutCenter} (see Figure~\ref{PictG1G2}), where $L$ and $R$ are the two components of $\mathbb{H}-V(M-u_1-w_1)$ that contain $u_1$ and $w_1$, respectively. Without loss of generality, say $\N(R)_{w_1}\geq \N(L)_{u_1}$. Then Lemma~\ref{Block.Two.cutCenter} shows the existence of another graph $G_3$ with order $n$ and $c$ cut vertices satisfying $\N(G_3)>\N(G_1)=\N(\mathbb{H})$, which is indeed a contradiction. 
\item $V(K_l)\cap V(K_r)\neq \emptyset$. Since $K_l$ and $K_r$ are both blocks of $\mathbb{H}$, they can only have one common vertex, which is therefore a cut vertex of $\mathbb{H}$. Thus, without loss of generality, say $V(K_l)\cap V(K_r)=\{u_1=w_1\}$. The graph $\mathbb{H}$ can then be given the same description as the graph $G_1$ defined in Lemma~\ref{Block.Two.cutCenterM=1}, where $x_2=u_2\in V(L_2)$ and $z_2=w_2\in V(R_2)$. Without loss of generality, say $\N(R_2)_{z_2}\geq \N(L_2)_{x_2}$. Then Lemma~\ref{Block.Two.cutCenterM=1} applied to $\mathbb{H}=G_1$, which contradicts the choice of $\mathbb{H}$.
\end{enumerate}
Summing up, we have proved that all blocks of $\mathbb{H}$, except possibly only one, are of order $2$. Moreover, every cut vertex of $\mathbb{H}$ belongs to precisely two distinct blocks of $\mathbb{H}$. This then makes it simple to derive the full structure of $\mathbb{H}$. It is easy to see that all blocks of $\mathbb{H}$ are of order $2$ if and only if $c=n-2$ ($\mathbb{H}$ is a path in this case). Assume that $c\leq n-3$ and let $K_q$ be the unique block of $\mathbb{H}$ such that $q>2$. One immediately deduces that $\mathbb{H}$ consists of $K_q$ to which $q$ paths (possibly trivial) $P_{n_1},\ldots, P_{n_q}$ are attached to the vertices $u_1,\ldots,u_q$ of $K_q$, respectively, by identifying $u_j$ with a leaf of $P_{n_j}$ for all $j$. Therefore, we have $(n_1-1)+\cdots + (n_q-1)=c$ and $n_1+\cdots + n_q=n$, i.e. $q=n-c$. To complete the proof of the theorem, we need to find the values of all $n_j$. Lemma~\ref{PathOrder} yields that $n_1,\ldots,n_q$ must all be as equal as possible, i.e.
\begin{align*}
n_1=\cdots =n_{n-c-t}=\lfloor n/(n-c) \rfloor=s~\text{and}~n_{n-c-t+1}=\cdots =n_{n-c}=s +1\,,
\end{align*}
where $t$ is the residue of $n$ modulo $n-c$. Hence, we get
\begin{align*}
\N(\mathbb{H})= (n-c-t)\binom{s}{2}+t\binom{s+1}{2}+(s+1)^{n-c-t}(s+2)^t -1
\end{align*}
as a special case in the proof of Lemma~\ref{PathOrder}. This completes the proof of the theorem.
\end{proof}

\subsection{The minimisation problem}
In this subsection, we consider the special case $c=0$ of the problem of finding those graphs that minimise the number of connected induced subgraphs among all graphs in the set $\mathcal{H}(n,c)$.

Let $n\geq 4$ and $G$ be a graph consisting of the cycle $C_{n-1}$ together with a vertex $z\notin V(C_{n-1})$ which is adjacent to precisely two vertices $x,v \in V(C_{n-1})$. In the sequel, we shall refer to every such graph as \emph{special}.

\begin{lemma}\label{SpeGraph}
If $G$ is a special graph of order $n$, then we have $\N(G)>n^2-n+1=\N(C_n)$.
\end{lemma}

\begin{proof}
Let $G$ be a special graph of order $n$. A simple lower bound on $\N(G)$ can be obtained as follows: a $z$-containing connected induced subgraph of $G$ is either the single vertex $z$, or consists of $z$ and at least a neighbour of $z$ in $G$. Thus, we get
\begin{align*}
\N(G)_z=1+\N(G-v)_{x,z}+\N(G-x)_{v,z}+\N(G)_{x,v,z}\,.
\end{align*}
Since $G-z$ is a cycle and $G-v-z$ as well as $G-x-z$ are paths, we deduce that $$\N(G)_z \geq 1+(n-2)+(n-2)+2=2n-1$$
from which the inequality
\begin{align*}
\N(G)=\N(G-z)+\N(G)_z \geq (n-1)(n-2)+1+2n-1 > n^2-n+1
\end{align*}
follows. 	
\end{proof}

A cut vertex-free connected graph with at least three vertices is also referred to as a $2$-connected graph. 
\begin{theorem}\label{Main2Cut}
For all $n\geq 3$, the cycle $C_n$ has the smallest number of connected induced subgraphs among all graphs in the set $\mathcal{H}(n,0)$.
\end{theorem}

\begin{proof}
Throughout the proof, it is assumed that $n\geq 3$. Let $G\in \mathcal{H}(n,0)$ be a graph of order $n$ that minimises the number of connected induced subgraphs. Then $G$ must necessary be minimally $2$-connected. In other words, $G$ must have the property that removing an edge in $G$ destroys $2$-connectivity. Moreover, in view of Lemma~\ref{SpeGraph}, $G$ cannot be a special graph. Suppose that $G$ is not a cycle. Clearly, we have $n\geq 5$. Let us prove that we can always identify $n^2+n+1>n^2-n+1=\N(C_n)$ connected induced subgraphs in $G$.

Let $u,v$ be two non-adjacent vertices of $G$. By the vertex version of Menger's theorem~\cite{menger1927allgemeinen}, there must exist two internally vertex disjoint paths between $u$ and $v$. Among all such $u-v$ paths, we choose two of them that are of smallest lengths. The vertex sets of these chosen paths must necessarily induce paths in $G$ since otherwise, the property of these paths being shortest is violated. Let $m$ denote the number of edges of $G$. Then the number of unordered pairs of non-adjacent vertices of $G$ is $\binom{n}{2}-m$, and therefore $G$ has at least $$2\Big(\binom{n}{2}-m \Big)=n^2-n-2m$$ connected induced subgraphs, each of them is a path of order three or more.

Let $x,y$ be two adjacent vertices of $G$. We claim that the graph $G-x-y$ is connected and moreover it is not a path. For the claim, suppose that $G-x-y$ is not connected and let $G_1,G_2$ be two (connected) components of $G-x-y$. Both $x$ and $y$ must have a neighbour in $G_1$ and $G_2$ because neither $x$ nor $y$ is a cut vertex of $G$. This implies that $G$ contains a cycle that passes through $x,y$ and never uses the edge $xy$. This cycle can be obtained as follows: let $x_1$ (resp. $x_2$) be a neighbour of $x$ in $G_1$ (resp. $G_2$), and $y_1$ (resp. $y_2$) be a neighbour of $y$ in $G_1$ (resp. $G_2$). Then this cycle is made of $xx_1$, a shortest $x_1-y_1$ path in $G_1$, $y_1y, yy_2$, a shortest $y_2-x_2$ path in $G_2$, and $x_2x$, in this order. However, by a result of Dirac~\cite[Theorem~3]{dirac1967minimally}, this cannot happen in a minimally $2$-connected graph. Hence, $G-x-y$ is connected. It remains to show that $G-x-y$ is not a path. Suppose to the contrary that $G-x-y$ is a path and let $u_1,u_2$ be the endvertices of $G-x-y$. Since $G$ is cut vertex-free, both $u_1$ and $u_2$ must have $x$ or $y$ as a neighbour. In a minimally $2$-connected graph, this gives rise to essentially two possibilities (up to exchanging the role of $x$ and $y$, or $u_1$ and $u_2$) for $G$: either $G$ itself is $C_n$, or $G$ consists of a cycle $C_{n-1}$ together with a vertex $z\notin V(C_{n-1})$ which is adjacent to precisely two vertices of $C_{n-1}$. The former situation is avoided by assumption while the latter defines $G$ as a special graph. Hence, $G-x-y$ is not a path. 

Let $x\in V(G)$. We further claim that the connected graph $G-x$ is not a path. To see this, note that if $G-x$ was a path, then its two endvertices would both be adjacent to $x$ since $G$ is $2$-connected. In particular $G$ would be a cycle since $G$ is minimally $2$-connected. Hence, $G-x$ is not a path.

Now we note that the following are all distinct connected induced subgraphs of $G$ and none of them is a path of order at least three:
\begin{itemize}
\item all single vertices of $G$;
\item all $2$-vertex connected subgraphs of $G$;
\item all subgraphs obtained by removing two adjacent vertices of $G$; 
\item all subgraphs obtained by removing one vertex of $G$;
\item the whole graph $G$.
\end{itemize}
By combining all the above cases, we obtain $n + m + m + n +1=2n+2m+1$ additional connected induced subgraphs of $G$ that are not paths of order three or more. Together with the induced paths enumerated earlier, we conclude that
\begin{align*}
\N(G)\geq n^2-n -2m+2n+2m+1=n^2+n+1>n^2-n+1=\N(C_n)\,.
\end{align*}
This completes the proof of the theorem.
\end{proof}

We observe that all blocks of a graph that minimises the number of connected induced subgraphs in the set $\mathcal{H}(n,c)$ must be minimally $2$-connected. However, there are usually many minimally $2$-connected graphs having the same order $n$. For $n\geq 3$, the sequence starts
\begin{align*}
1, 1, 2, 3, 6, 12, 28, 68, 184, 526, 1602, 5075, 16711, 56428, 195003, 685649, \ldots\,,
\end{align*}
see \url{A003317} in~\cite{sloaneoeis}. It is then natural to formulate this intriguing problem for further investigation:

\begin{problem}
Find a constructive characterisation of those graphs with order $n$ and $c>0$ cut vertices that have the smallest number of connected induced subgraphs.
\end{problem}

\section{Connected graphs with $p$ pendant vertices}\label{connected p pendant}
We define $\mathcal{G}(n,p)$ to be the set of all connected graphs with $n$ vertices of which $p$ are pendant. In~\cite{EricAudace} Andriantiana and the author of the present paper investigated inequalities which relate the number of connected induced subgraphs of a graph to that of its complement. They also arrived at the following result which settles the extremal graph structure for the maximum number of connected induced subgraphs among all graphs in $\mathcal{G}(n,p)$.

\begin{theorem}[\cite{EricAudace}]\label{maxn.p}
Let $G \in \mathcal{G}(n,p)$ with $n\geq 5$ and $0\leq p \leq n-2$.
\begin{itemize}
\item If $p< n-2$, then we have 
\begin{equation*} \label{Eq:Mk2}
\N(G) \leq 2^{n-1} + 2^{n-p-1} + p -1\,.
\end{equation*}
Equality happens if and only if $G$ can be obtained by identifying one vertex of $K_{n-p}$ with the central vertex of $S_{p+1}$. 
\item If $p=n-2$, then we have
\begin{equation*}
\N(G) \leq n+3\cdot 2^{n-3}\,.
\end{equation*}
Equality happens if and only if $G$ can be obtained by inserting one vertex into an edge of $S_{n-1}$.
\end{itemize}
\end{theorem}

In order to obtain the minimisation counterpart of Theorem~\ref{maxn.p}, we need to state two intermediate results.

\subsection{The case $p\neq 0$}
Sharp bounds on the number of connected induced subgraphs in terms of order were obtained in~\cite{Audacegenral2018}. One of the results in~\cite{Audacegenral2018} will be needed for our purpose.

\begin{theorem}[\cite{Audacegenral2018}--Theorem~9]\label{min.Uni.n}
If $G$ is a unicylic graph of order $n$, then $$N(G) \geq (n^2 + 3n - 4)/2.$$ The bound is attained if and only if $G$ can be obtained by identifying a vertex of $K_3$ with a leaf of $P_{n-2}$.
\end{theorem} 

At this stage, recall that $\mathbb{T}^1(n,p)$ is the tree obtained from the vertex disjoint graphs $S_{1+\lfloor p/2 \rfloor}$ and $S_{1+\lceil p/2 \rceil}$ by identifying their central vertices with the two leaves of $P_{n-p}$, respectively.
We recall Li and Wang's result as stated in the introduction. 

\begin{theorem}[\cite{LiWang}--Theorem~1]\label{min.tree.leaf}
If $n\geq 4$ and $2\leq p \leq n-2$, then $\mathbb{T}^1(n,p)$ is the unique tree with order $n$ and $p$ leaves that attains the minimum number of subtrees.
\end{theorem}

Our next theorem, which is essentially extracted from Theorem~\ref{min.tree.leaf}, reads as follows:

\begin{theorem}\label{MinPpendGrap}
Let $G \in \mathcal{G}(n,p)$ with $n\geq 4$ and $1\leq p \leq n-2$.
\begin{itemize}
\item If $p=1$, then we have
\begin{equation*}
\N(G) \geq (n^2 + 3n - 4)/2\,.
\end{equation*}
Equality happens if and only if $G$ can be obtained by identifying a vertex of $K_3$ with a leaf of $P_{n-2}$.
\item If $p\neq 1$, then we have 
\begin{equation*} 
\N(G) \geq 2^p +(n-p-1)(2^{\lfloor p/2 \rfloor}+2^{\lceil p/2 \rceil}) + p+(n-p-1)(n-p-2)/2\,.
\end{equation*}
Equality happens if and only if $G$ is isomorphic to $\mathbb{T}^1(n,p)$.
\end{itemize}
\end{theorem}

\begin{proof}
Suppose that $p=1$. Then $G$ is not a tree. One can then remove edges (possibly none) from $G$ to get a new connected graph $G'$ of order $n$ that contains exactly one cycle. It follows from Theorem~\ref{min.Uni.n} that $\N(G)\geq \N(G')\geq (n^2 + 3n - 4)/2 $. Moreover, the unique graph attaining this bound also has exactly one pendant vertex as it can be obtained by identifying a vertex of $K_3$ with a leaf of $P_{n-2}$. Thus the result follows in this case.

Now suppose that $p\neq 1$. Let us first derive a formula for $\N(\mathbb{T}^1(n,p))$. Denote by $u$ and $v$ the two vertices of $\mathbb{T}^1(n,p)$ such that $\deg_{\mathbb{T}^1(n,p)} (u)=1+\lfloor p/2 \rfloor$ and $\deg_{\mathbb{T}^1(n,p)} (v)=1+\lceil p/2 \rceil$. We have
\begin{align*}
\N(\mathbb{T}^1(n,p))&=\N(\mathbb{T}^1(n,p))_{u,v}+\N(\mathbb{T}^1(n,p)-v)_u+\N(\mathbb{T}^1(n,p)-u)_v+\N(\mathbb{T}^1(n,p)-u-v)\\
&=2^{\lfloor p/2 \rfloor + \lceil p/2 \rceil} +(n-p-1)2^{\lfloor p/2 \rfloor}+(n-p-1)2^{\lceil p/2 \rceil}\\
&+ \big(\lfloor p/2 \rfloor +(n-p-1)(n-p-2)/2+ \lceil p/2 \rceil\big)\,.
\end{align*}
We claim that $\N(\mathbb{T}^1(n,p))$ is an increasing function in $p$. Indeed, we have
\begin{align*}
\N(\mathbb{T}^1(n,p))&=2^p +(n-p-1)(2^{p/2}+2^{p/2})+p+\binom{n-p-1}{2}\,,\\
\N(\mathbb{T}^1(n,p+1))&=2^{p+1} +(n-p-2)(2^{p/2}+2^{p/2+1})+p+1+\binom{n-p-2}{2}
\end{align*}
if $p$ is even, and
\begin{align*}
\N(\mathbb{T}^1(n,p))&=2^p +(n-p-1)(2^{(p-1)/2}+2^{(p+1)/2})+p+\binom{n-p-1}{2}\,,\\
\N(\mathbb{T}^1(n,p+1))&=2^{p+1} +(n-p-2)(2^{(p+1)/2}+2^{(p+1)/2})+p+1+\binom{n-p-2}{2}
\end{align*}
if $p$ is odd. In particular, we get
\begin{align*}
\N(\mathbb{T}^1(n,p+1))-\N(\mathbb{T}^1(n,p))=2^p-1+(n-p-4)(2^{p/2}-1)\geq 1+2^p-2^{p/2+1}>0
\end{align*}
if $p$ is even, and
\begin{align*}
\N(\mathbb{T}^1(n,p+1))-\N(\mathbb{T}^1(n,p))&=2^p-2+(n-p-5)(2^{(p-1)/2}-1)\\
&\geq 1+2^p-2^{(p-1)/2}-2^{(p+1)/2}>0
\end{align*}
if $p$ is odd. Let $T$ be a spanning tree of $G$ and note that $T$ has at least $p$ leaves. Since $\N(\mathbb{T}^1(n,p))$ is an increasing function in $p$, we deduce from Theorem~\ref{min.tree.leaf} that
\begin{align*}
\N(G)\geq \N(T) \geq \N(\mathbb{T}^1(n,p(T))) > \N(\mathbb{T}^1(n,p)) 
\end{align*}
if $p(T)\neq p$. If $p(T)= p$, then we have
\begin{align*}
\N(G)\geq \N(T) \geq  \N(\mathbb{T}^1(n,p))\,, 
\end{align*}
and the inequality becomes an equality if and only if $G$ is isomorphic to the tree $\mathbb{T}^1(n,p)$. This completes the proof of the theorem.
\end{proof}

\subsection{The case $p=0$}
Let $l,n,r$ be three positive integers such that $l,r\geq 3$ and $n\geq l+r$. We define the \emph{double tadpole} graph $D_n(l;r)$ as the graph constructed from the three pairwise vertex disjoint graphs $C_l,C_r,P_{n+2-l-r}$ by taking $u\in V(C_l),~v\in V(C_r)$ and identifying $u$ with one leaf of $P_{n+2-l-r}$ and $v$ with the other leaf of $P_{n+2-l-r}$.

For $n>5$, we shall prove that the double tadpole graph $D_n(3;3)$ has the smallest number of connected induced subgraphs among all graphs in the set $\mathcal{G}(n,0)$, and that $D_n(3;3)$ is unique with this property.

\medskip
We first give some important preliminaries, then formally state and prove our result. From here onwards, we shall simply write $D_n$ instead of $D_n(3;3)$.
\begin{proposition}\label{PropTadpole}
	For the double tadpole graph $D_n(l;r)$, we have
	\begin{align*}
	\N(D_n(3;3))= \N(D_n)=\frac{(n-1)(n+6)}{2}\,.
	\end{align*}
	Furthermore, if $(l,r)\neq (3,3)$, then we have
	\begin{align*}
	\N(D_n(l;r)) > \N(D_n)\,.
	\end{align*}
\end{proposition}

\begin{proof}
	Let $u,v \in V(D_n(l;r))$ be the two vertices of $D_n(l;r)$ whose degree is $3$. We use our standard decomposition with respect to $u,v$:
	\begin{align*}
	\N(D_n(l;r))&=\N(D_n(l;r))_{u,v}+\N(D_n(l;r)-v)_u+ \N(D_n(l;r)-u)_v+ \N(D_n(l;r)-u-v)\\
	&=\N(C_l)_u \cdot \N(C_r)_v +(n+1-l-r) (\N(C_l)_u + \N(C_r)_v)\\
	&+ \N(P_{l-1}) + \N(P_{n-l-r})+ \N(P_{r-1})\\
	& =\big(1+ \binom{l}{2}\big)\N(C_r)_v + (n+1-l-r) \big(1+ \binom{l}{2} + \N(C_r)_v \big)\\
	& + \binom{l}{2} + \binom{n+1-l-r}{2} + \N(P_{r-1}) \,.
	\end{align*}
	Assume that $r\geq l\geq 4$. Taking the difference $\N(D_n(l;r))-\N(D_n(l-1;r))$, we get
	\begin{align*}
	\N(D_n(l;r))&- \N(D_n(l-1;r))=(l-1) \N(C_r)_v  + (l-1)(n+1-l-r)\\
	& - \big(1+ \binom{l-1}{2}+ \N(C_r)_v \big) + (l-1) - (n+1-l-r)\\
	&=(l-2) \N(C_r)_v - \N(C_l)_u + (l-2)(n+1-l-r) + 2(l-1)>0
	\end{align*}
	since $\N(C_r)_v \geq \N(C_l)_u $ and $n\geq l+r$.  It follows from this inequality that the minimum of $\N(D_n(l;r))$, given $r$, is attained when $l=3$. We have
	\begin{align*}
	\N(D_n(3;3))=\N(D_n) =22+ 8(n-5) + \binom{n-5}{2}=\frac{(n-1)(n+6)}{2}\,.
	\end{align*}
	Assume that $l=3$ and $r\geq 4$. Taking the difference $\N(D_n(3;r))-\N(D_n)$, we get
	\begin{align*}
	\N(D_n(3;r))&- \N(D_n)=4\cdot \N(C_r)_v + (n-r-2) \big(4 + \N(C_r)_v \big)\\
	& + 3 + \binom{n-r-2}{2} + \binom{r}{2} - \frac{(n-1)(n+6)}{2} \\
	&=\frac{(r-3)(-r^2+(n+2)r -2)}{2}>0
	\end{align*}
	after a simple manipulation (recall that $n\geq 3+r$). It follows from this inequality that the minimum of $\N(D_n(3;r))$ is attained when $r=3$.
\end{proof}

\begin{lemma}\label{ordCnandDn}
	If $G$ is a graph constructed from two vertex disjoint cycles $C_l$ and $C_{n-l+1}$ by identifying $u\in V(C_l)$ with $v \in V(C_{n-l+1})$, then we have
	\begin{align*}
	\N(G)>N(C_n)>\N(D_n)
	\end{align*}
	for all $n>5$.
\end{lemma}

\begin{proof}
	Simply note that
	\begin{align*}
	\N(G)=\N(C_l)_u \cdot \N(C_{n-l+1})_v+ \N(P_{l-1}) +  \N(P_{n-l})\,,
	\end{align*}
	and that the difference $\N(G)-\N(C_n)$ is given by
	\begin{align*}
	\N(G)-\N(C_n) &=\big(1+\binom{l}{2}\big)\big(1+\binom{n-l+1}{2}\big)+ \binom{l}{2}+ \binom{n-l+1}{2} - (n^2-n+1)\\
	&=\frac{(l-1)(l-n)(l^2-(n+1)l +8)}{4} >0
	\end{align*}
	as $3\leq l \leq n-2$. Likewise, the difference $\N(C_n)-\N(D_n)$ is given by
	\begin{align*}
	\N(C_n)-\N(D_n) &= (n^2-n+1) - \frac{(n-1)(n+6)}{2}\\
	&=\frac{n^2-7n+8}{2}>0\,.
	\end{align*}
\end{proof}

\begin{lemma}\label{qksliding} 
	Let $L, R$ be two fixed non-trivial vertex disjoint connected graphs such that $u\in V(L)$ and $v \in V(R)$. Consider two vertex disjoint paths $P_k,P_q$ for some $q\geq 2$. Identify $u$ with both a leaf of $P_k$ as well as a leaf of $P_q$; further, identify $v$ with the other leaf of $P_q$. Denote by $H(k;q)$ the resulting graph. If $k> 1$, then we have
	\begin{align*}
	\N(H(k;q)) > \N(H(1;q+k-1))\,.
	\end{align*}
	
	\begin{proof}
		We use our standard decomposition again:
		\begin{align*}
		\N(H(k;q))&=\N(H(k;q))_{u,v}+\N(H(k;q)-v)_u + \N(H(k;q)-u)_v + \N(H(k;q)-u-v)\\
		&=k\cdot \N(L)_u \cdot \N(R)_v + (q-1) (k\cdot \N(L)_u  + \N(R)_v )\\
		&+ \N(L-u)+\binom{k}{2} +\binom{q-1}{2} + \N(R-v)\,,
		\end{align*}
		and
		\begin{align*}
		\N(H(1;q+k-1))&=\N(L)_u \cdot \N(R)_v + (q+k-2)(\N(L)_u  + \N(R)_v )\\
		&+ \N(L-u) +\binom{q+k-2}{2} + \N(R-v)\,.
		\end{align*}
		It follows that
		\begin{align*}
		\N(H(k;q)) &- \N(H(1;q+k-1))=(k-1)(\N(L)_u -1) \N(R)_v \\
		& + (k-1)(q-2)\N(L)_u +\binom{k}{2} +\binom{q-1}{2}  - \binom{q+k-2}{2}\,.
		\end{align*}
		On the other hand, we have
		\begin{align*}
		2(k-1)(q-2) +\binom{k}{2} +\binom{q-1}{2}  - \binom{q+k-2}{2}=(k-1)(q-2)\geq 0\,.
		\end{align*}
		Therefore, using the assumption that $\N(L)_u\geq 2$, we derive that
		\begin{align*}
		\N(H(k;q)) - \N(H(1;q+k-1))\geq (k-1)(\N(L)_u -1) \N(R)_v  + (k-1)(q-2)>0
		\end{align*}
		provided that $k\neq 1$. This completes the proof.
	\end{proof}
\end{lemma}

Our next lemma captures the special case $q=1$ that is missing in Lemma~\ref{qksliding}.

\begin{lemma}\label{q=1ksliding}
	Let $L, R$ be two fixed non-trivial vertex disjoint connected graphs such that $u\in V(L)$ and $v \in V(R)$. Consider the path $P_k, k\geq 2$ and let $w$ be a leaf of $P_k$. Identify $w$ with both $u$ and $v$. Denote by $H(k;1)$ the resulting graph. Then we have
	\begin{align*}
	\N(H(k;1)) > \N(H(1;k))\,,
	\end{align*}
	where $H(1;k)$ is the graph described in Lemma~\ref{qksliding}.
\end{lemma}	
\begin{proof}
	Simply note that
	\begin{align*}
	\N(H(k;1))=k\cdot \N(L)_u \cdot \N(R)_v +\N(L-u)+\binom{k}{2} +\N(R-v)\,,
	\end{align*}
	and that
	\begin{align*}
	\N(H(1;k))&=\N(L)_u \cdot \N(R)_v + (k-1) (\N(L)_u  + \N(R)_v )\\
	& + \N(L-u)+\binom{k-1}{2} + \N(R-v)\,.
	\end{align*}
	In particular, we get
	\begin{align*}
	\N(H(k;1)) - \N(H(1;k))=(k-1)(\N(L)_u -1)(\N(R)_v -1)>0\,.
	\end{align*}
\end{proof}

The following lemma is a variant of the combination of Lemmas~\ref{qksliding} and~\ref{q=1ksliding}. 
\begin{lemma}\label{Refqksliding}
	Let $L, R$ be two fixed non-trivial vertex disjoint connected graphs such that $u,w\in V(L),~u\neq w$ and $v \in V(R)$. Consider three vertex disjoint paths $P_k,P_q,P_{q+k-1}$ for some $k>1$. Let $G_1,G_2$ be the two graphs constructed as follows:
	\begin{itemize}
	\item If $q>1$, then identify $w$ with a leaf of $P_k$, $u$ with a leaf of $P_q$, and $v$ with the other leaf of $P_q$ to obtain $G_1$. If $q=1$, then identify $w$ with a leaf of $P_k$, and $u$ with $v$ to obtain $G_1$. 
	\item Identify $u$ with a leaf of $P_{q+k-1}$, and $v$ with the other leaf of $P_{q+k-1}$ to obtain $G_2$.
	\end{itemize}
	We have 
	\begin{align*}
	|V(G_1)|=|V(G_2)| \quad \text{and} \quad \N(G_1)>\N(G_2)\,.
	\end{align*}

\begin{proof}
Denote by $J$ the subgraph of $G_1$ that consists of $L$ and a leaf of $P_k$ attached to $L$ at vertex $w$. Assume that $q>1$. Then we have
		\begin{align*}
		\N(G_1)&=\N(G_1)_{u,v}+\N(G_1-v)_u + \N(G_1-u)_v + \N(G_1-u-v)\\
		&=\N(J)_u \cdot \N(R)_v + (q-1) (\N(J)_u  + \N(R)_v )\\
		&+ \N(J-u) +\binom{q-1}{2} + \N(R-v)\,,
		\end{align*}
		and
\begin{align*}
		\N(G_2)&=\N(L)_u \cdot \N(R)_v + (q+k-2)(\N(L)_u  + \N(R)_v )\\
		&+ \N(L-u) +\binom{q+k-2}{2} + \N(R-v)\,.
		\end{align*}
In particular, we get		
\begin{align*}
\N(G_1)-\N(G_2)&= \N(R)_v(\N(J)_u - \N(L)_u +1-k)  +(q-1)\N(J)_u -(q+k-2)\N(L)_u \\
&+ \N(J-u) - \N(L-u) +\binom{q-1}{2} -\binom{q+k-2}{2}\,.
\end{align*}		
On the other hand, we have
\begin{align*}
\N(J)_u \geq \N(L)_u + k\cdot \N(L)_{u,w} \geq \N(L)_u + k
\end{align*}
and
\begin{align*}
\N(J-u) \geq \N(L-u) + \binom{k}{2} + k\cdot \N(L-u)_w \geq \N(L-u) + \binom{k}{2} + k\,.
\end{align*}
This implies that
\begin{align*}
\N(G_1)-\N(G_2)&\geq \N(R)_v  +(q-1)(k+\N(L)_u)  -(q+k-2)\N(L)_u +k \\
& +\binom{q-1}{2} +\binom{k}{2} -\binom{q+k-2}{2}\\
&=\N(R)_v  +k\cdot q + (k-1)(q-2) \N(L)_u \\
& +\binom{q-1}{2} +\binom{k}{2} -\binom{q+k-2}{2}\,.
\end{align*}
It follows from the identity
\begin{align*}
2(k-1)(q-2) +\binom{k}{2} +\binom{q-1}{2}  - \binom{q+k-2}{2}=(k-1)(q-2)
\end{align*}
that
\begin{align*}
\N(G_1)-\N(G_2)&\geq \N(R)_v  +k\cdot q + (k-1)(q-2) >0\,.
\end{align*}
Assume that $q=1$. Then we have
\begin{align*}
\N(G_1)&=\N(G_1)_u+\N(G_1-u)=\N(J)_u \cdot \N(R)_v + \N(J-u) + \N(R-v)\\
&\geq \N(R)_v (\N(L)_u + k)+ \N(L-u) + \binom{k}{2} + k + \N(R-v) \,,
\end{align*}
and
\begin{align*}
\N(G_2)&=\N(L)_u \cdot \N(R)_v + (k-1)(\N(L)_u  + \N(R)_v ) + \N(L-u) +\binom{k-1}{2} + \N(R-v)\\
&\geq \N(L)_u \cdot \N(R)_v + (k-1)(2  + \N(R)_v ) + \N(L-u) +\binom{k-1}{2} + \N(R-v)\,.
\end{align*}
In particular, we get
\begin{align*}
\N(G_1)-\N(G_2) \geq \N(R)_v + \binom{k}{2} -\binom{k-1}{2} + k - 2(k-1) =1+ \N(R)_v>0\,.
\end{align*}
This completes the proof of the lemma.
\end{proof}
\end{lemma}

We finish our preliminaries with the following lemma, which is similar in nature but different to Lemma~\ref{PathOrder} (see Section~\ref{connected c cut}).

\begin{lemma}\label{MINPathOrder}
Let $H$ be a connected graph of order greater than two, and $u,v$ two distinct vertices of $H$ such that $\N(H)_{u,v}>1$ and $\N(H-v)_u \leq  \N(H-u)_v$. Let $H(n_1;n_2)$ be the graph obtained from $H$ by identifying $u$ with a leaf of $P_{n_1}$, and $v$ with a leaf of $P_{n_2}$ for some $n_1,n_2\geq 1$. We have
\begin{align*}
	\N(H(n_1;n_2)) \geq \N(H(n_1+n_2-1;1))\,.
	\end{align*}
	Moreover, the inequality is strict if $n_1,n_2>1$.
\end{lemma}

\begin{proof}
	By the proof of Lemma~\ref{PathOrder}, we have
	\begin{align*}
	\N(H(n_1;n_2))&=n_1\cdot n_2\cdot \N(H)_{u,v}+n_1\cdot \N(H-v)_u + n_2\cdot \N(H-u)_v\\
	&+\N(P_{n_1-1})+ \N(P_{n_2-1})+\N(H-u-v)\,.
	\end{align*}
	In particular, we get
	\begin{align*}
	\N(H(n_1+n_2-1;1))&=(n_1+n_2-1) \N(H)_{u,v}+(n_1+n_2-1) \N(H-v)_u +  \N(H-u)_v\\
	&+\N(P_{n_1+n_2-2})+\N(H-u-v)\,,
	\end{align*}
	which implies that
	\begin{align*}
	\N(H(n_1+&n_2-1;1))-\N(H(n_1;n_2))=(n_1+n_2-1-n_1\cdot n_2)\N(H)_{u,v}\\
	& + (n_2-1) \N(H-v)_u +(1-n_2) \N(H-u)_v\\
	&+\binom{n_1+n_2-1}{2} - \binom{n_1}{2} - \binom{n_2}{2}\\
	&=(n_2-1)(\N(H-v)_u -  \N(H-u)_v) -(n_1-1)(n_2-1)(\N(H)_{u,v} -1) \leq 0\,.
	\end{align*}
	Moreover, we have $\N(H(n_1+n_2-1;1))< \N(H(n_1;n_2))$ if $n_1>1$ and $n_2>1$. The statement of the lemma follows.
\end{proof}	

By rooted path, we mean a path rooted at one of its leaves. Our main result reads as follows:

\begin{theorem}\label{Theo:p=0}
	Let $n>5$ be a positive integer. For every graph $G \in  \mathcal{G}(n,0)$, we have
	\begin{align*}
	\N(G)\geq \N(D_n)=  \frac{(n-1)(n+6)}{2}\,,
	\end{align*}
	and $D_n \in  \mathcal{G}(n,0)$ is the only graph with this property.
\end{theorem}

\begin{proof}
Let $\mathbb{G} \in  \mathcal{G}(n,0)$ be a connected graph with order $n$ and no pendant vertex that minimises the number of connected induced subgraphs. We are going to show that $\mathbb{G}$ can be obtained from certain graphs $H_1 \in \mathcal{G}(n,0)$ through a series of graph transformations that preserve the number of vertices.

First off, note that $\mathbb{G}$ must have at least one cut vertex, since otherwise $\N(\mathbb{G})\geq \N(C_n)$ by virtue of Theorem~\ref{Main2Cut}, while Lemma~\ref{ordCnandDn} implies that $\N(C_n)>\N(D_n)$. Fix $H_1 \in \mathcal{G}(n,0)$ such that $H_1$ has at least one cut vertex. If we remove edges from a graph, the number of connected induced subgraphs decreases. Starting from $H_1$, we can thus remove certain edges until we reach a connected graph with only two distinct cyclic blocks, say $B_1,B_2$. More precisely, all blocks, except only two of $H_1$ are replaced with any generic of their spanning trees. This yields a new graph $H_2$ which may contain a pendant vertex. Moreover, we have $\N(H_1)\geq \N(H_2)$.

In the graph $H_2$, we can remove edges from the blocks $B_1,B_2$ in such a way that the two cyclic blocks of the resulting graph, say $H_3$ are all cycles, say $C_l$ and $C_r$. Hence, $H_3$ consists of two distinct cycles $C_l,C_r$ `separated' by a (possibly trivial) path $P$, together with some trees attached to all vertices of $V(C_l)\cup V(C_r)\cup V(P)$ in $H_3$; see Figure~\ref{Picp=0} for a picture. Moreover, we have $\N(H_2)\geq \N(H_3)$.
	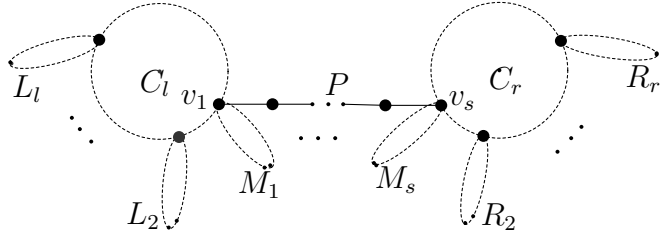
\begin{figure}[!h]
		\definecolor{tttttt}{rgb}{0.2,0.2,0.2}
		\begin{tikzpicture}[scale=1.5,line cap=round,line join=round,>=triangle 45,x=1.0cm,y=1.0cm]
		
		\draw [dash pattern=on 1pt off 1pt] (2.01,5.59) circle (0.6cm);
		\draw [dash pattern=on 1pt off 1pt] (4.99,5.6) circle (0.6cm);
		\draw [rotate around={84.19:(2.14,4.61)},dash pattern=on 1pt off 1pt] (2.14,4.61) ellipse (0.41cm and 0.1cm);
		\draw [rotate around={15:(1.09,5.76)},dash pattern=on 1pt off 1pt] (1.09,5.76) ellipse (0.41cm and 0.1cm);
		\draw [rotate around={78.54:(4.77,4.63)},dash pattern=on 1pt off 1pt] (4.77,4.63) ellipse (0.4cm and 0.09cm);
		\draw [rotate around={-8.63:(5.95,5.8)},dash pattern=on 1pt off 1pt] (5.95,5.8) ellipse (0.43cm and 0.09cm);
		\draw (1.74,5.7) node[anchor=north west] {$C_l$};
		\draw (4.8,5.72) node[anchor=north west] {$C_r $};
		\draw (2.1,5.52) node[anchor=north west] {$ v_1 $};
		\draw (4.45,5.51) node[anchor=north west] {$ v_s $};
		\draw (1.62,4.53) node[anchor=north west] {$L_2 $};
		\draw (0.62,5.65) node[anchor=north west] {$L_l $};
		\draw (4.73,4.52) node[anchor=north west] {$ R_2 $};
		\draw (6.,5.74) node[anchor=north west] {$ R_r $};
		\draw [rotate around={-50.31:(2.76,5.03)},dash pattern=on 1pt off 1pt] (2.76,5.03) ellipse (0.37cm and 0.11cm);
		\draw [rotate around={41.28:(4.18,5.04)},dash pattern=on 1pt off 1pt] (4.18,5.04) ellipse (0.4cm and 0.1cm);
		\draw (2.53,5.3)-- (3,5.3);
		\draw (4.48,5.29)-- (3.99,5.29);
		\draw (3,5.3)-- (3.35,5.3);
		\draw (3.99,5.29)-- (3.62,5.3);
		
		\fill [color=black] (3.25,5.) circle (0.5pt);
		\fill [color=black] (3.4,5.) circle (0.5pt);
		\fill [color=black] (3.55,5.) circle (0.5pt);
		
		\draw (2.6,4.80) node[anchor=north west] {$ M_1 $};
		\draw (3.8,4.85) node[anchor=north west] {$ M_s$};
		\draw (3.36,5.67) node[anchor=north west] {$ P $};
		
		\fill [color=tttttt] (2.01,5.59) circle (0.5pt);
		\fill [color=black] (4.99,5.6) circle (0.5pt);
		\fill [color=black] (2.53,5.3) circle (1.5pt);
		\fill [color=black] (4.48,5.29) circle (1.5pt);
		\fill [color=tttttt] (2.18,5.01) circle (1.5pt);
		\fill [color=black] (1.48,5.87) circle (1.5pt);
		\fill [color=black] (4.85,5.02) circle (1.5pt);
		\fill [color=black] (5.53,5.86) circle (1.5pt);
		\fill [color=black] (2.09,4.21) circle (0.5pt);
		\fill [color=black] (2.16,4.27) circle (0.5pt);
		\fill [color=black] (3.62,5.3) circle (0.5pt);
		\fill [color=black] (3.35,5.3) circle (0.5pt);
		\fill [color=tttttt] (0.71,5.66) circle (0.5pt);
		\fill [color=black] (0.7,5.67) circle (0.5pt);
		\fill [color=black] (4.69,4.25) circle (0.5pt);
		\fill [color=black] (4.76,4.31) circle (0.5pt);
		\fill [color=black] (6.37,5.74) circle (0.5pt);
		\fill [color=black] (6.37,5.75) circle (0.5pt);
		\fill [color=black] (1.24,5.18) circle (0.5pt);
		\fill [color=black] (1.31,5.08) circle (0.5pt);
		\fill [color=black] (1.41,4.97) circle (0.5pt);
		\fill [color=black] (5.71,5.1) circle (0.5pt);
		\fill [color=black] (5.5,4.88) circle (0.5pt);
		\fill [color=black] (2.98,4.76) circle (0.5pt);
		\fill [color=black] (2.93,4.74) circle (0.5pt);
		\fill [color=black] (3.49,5.3) circle (0.5pt);
		\fill [color=black] (3.89,4.78) circle (0.5pt);
		\fill [color=black] (3.91,4.76) circle (0.5pt);
		\fill [color=black] (3,5.3) circle (1.5pt);
		\fill [color=black] (3.99,5.29) circle (1.5pt);
		\fill [color=black] (5.62,5) circle (0.5pt);
		\end{tikzpicture}
		\caption{The graph $H_3$ in the proof of Theorem~\ref{Theo:p=0}: $P$ is a path starting at $v_1$ and ending at $v_s$; $L_2,\ldots,L_l,M_1,\ldots,M_s,R_2,\ldots,R_r$ are all trees.}\label{Picp=0}
	\end{figure}

In the graph $H_3$, replace all components $C$ of $H_3-(E(C_l)\cup E(C_r) \cup E(P))$ with a rooted path of order $|V(C)|$ rooted at the unique vertex of $C$ that belongs to $V(C_l)\cup V(C_r) \cup V(P)$. This gives us a new graph $H_4$. We claim that $\N(H_3)\geq \N(H_4)$ with equality if and only if $H_3$ and $H_4$ are isomorphic. Indeed, construct from $H_3$ a new graph $H_3'$ by replacing (without loss of generality) $M_1$ with the rooted path $P_{|V(M_1)|}$ whose root is $v_1$. Thus, $H_3-V(M_1-v_1)$ and $H_3'-V(M_1-v_1)$ are isomorphic graphs. On the other hand, if $A$ denotes the number of connected induced subgraphs of $H_3-V(M_1-v_1)=H_3'-V(M_1-v_1)$ that contain $v_1$, then
\begin{align*}
\N(M_1-v_1)+A\cdot \N(M_1)_{v_1}~ (\text{resp.}~\N(P_{|V(M_1)|-1})+A\cdot \N(P_{|V(M_1)|})_{v_1})
\end{align*}
counts precisely the number of connected induced subgraphs of $H_3$ (resp. $H_3'$) that contain a vertex of $M_1-v_1$. Since the path $P_m$ (rooted at a leaf) minimises both the total number of subtrees and the number of subtrees containing a specific vertex $u$ among all $m$-vertex trees (see Sz{\'e}kely and Wang~\cite{szekely2005subtrees}), we deduce that $\N(P_{|V(M_1)|-1}) \leq \N(M_1-v_1)$ and $\N(P_{|V(M_1)|})_{v_1} \leq \N(M_1)_{v_1}$. This implies that
\begin{align*}
\N(H_3')-\N(H_3)=\N(P_{|V(M_1)|-1}) - \N(M_1-v_1) + A( \N(P_{|V(M_1)|})_{v_1} - \N(M_1)_{v_1} ) \leq 0\,.
\end{align*}
Hence, we have $\N(H_3')\leq \N(H_3)$. Equality holds if and only if $M_1$ is a rooted path (see Sz{\'e}kely and Wang~\cite{szekely2005subtrees}), i.e. $H_3$ and $H_3'$ are isomorphic graphs. Since $H_4$ can be obtained from $H_3$ by a repetitive application of this process of moving from $H_3$ to $H_3'$, we derive that $\N(H_4)\leq  \N(H_3)$ with equality if and only if  $H_3$ and $H_4$ are isomorphic.

\medskip
In the graph $H_4$, fix two distinct vertices $u,v \in V(C_l)\cup V(C_r) \cup V(P)$, and consider $H_4$ as the graph $H(n_1;n_2)$ described in Lemma~\ref{MINPathOrder} where $n_1$ (resp. $n_2$) is the order of the path attached at $u$ (resp. $v$) in $H_4$. Lemma~\ref{MINPathOrder} states that whenever $n_1>1$ or $n_2>1$, two new graphs $H(n_1+n_2-1;1)$ and $H(1;n_1+n_2-1)$ can always be constructed from $H(n_1;n_2)$ such that at least one of the inequalities
\begin{align*}
\N(H(n_1;n_2)) > \N(H(n_1+n_2-1;1)) \quad \text{and} \quad  \N(H(n_1;n_2))> \N(H(1;n_1+n_2-1))
\end{align*}
holds. In other words, this shows that a graph $H_5$ with order $n$ and the property that $\N(H_4)\geq \N(H_5)$, can be obtained from $H_4$ by making all components (paths), except possibly only one of $H_4-(E(C_l)\cup E(C_r) \cup E(P))$ trivial. This leaves $H_5$ with two possible shapes if $H_5-(E(C_l)\cup E(C_r) \cup E(P))$ has a non trivial component (rooted path), say $P_k$:
\begin{itemize}
\item Vertex $v_1$ or $v_s$ is the root of $P_k$. In this case, we invoke Lemma~\ref{q=1ksliding} or Lemma~\ref{qksliding} on $H_5$ depending on whether the path $P$ `separating' the cycles $C_l$ and $C_r$ in $H_5$ is trivial or not; 
\item Neither $v_1$ nor $v_s$ is the root of $P_k$. In this case, we apply Lemma~\ref{Refqksliding}.
\end{itemize}
In either case, the combination of Lemmas~\ref{qksliding},~\ref{q=1ksliding}, and~\ref{Refqksliding} shows the existence of another graph $H_6 \in \mathcal{G}(n,0)$ with the property that $\N(H_5)\geq \N(H_6)$ with equality if and only if $H_5$ and $H_6$ are isomorphic. Moreover, by construction $H_6$ is either a double tadpole graph, or a graph constructed from two vertex disjoint cycles $C_l$ and $C_{n-l+1}$ by identifying $u\in V(C_l)$ with $v \in V(C_{n-l+1})$. The latter situation corresponds to the graph described in Lemma~\ref{ordCnandDn}. Consequently, $H_6$ can only be a double tadpole graph if $\N(H_6)$ is to be the minimum number of connected induced subgraphs that a connected graph with order $n$ and no pendant vertices can have.
	
Finally, we invoke Proposition~\ref{PropTadpole} on $H_6$ to obtain the double tadpole graph $D_n$ which satisfies $\N(H_6)>\N(D_n)$ provided that $H_6 \neq D_n$. Summing up, we have proved that  $\mathbb{G}$ is indeed the double tadpole graph $D_n$.
\end{proof}

\section*{Acknowledgement}
The author is pleased to acknowledge a discussion with Stephan Wagner.

\end{document}